\documentclass[10pt,draftcls,onecolumn]{IEEEtran}

\usepackage{float}
\usepackage{amsmath,amsbsy,amsgen,amscd,amsthm,amsfonts,amssymb} 
\usepackage{esint}
\usepackage{url}
\usepackage[usenames,dvipsnames]{color}
\usepackage{enumerate}
\usepackage{ifthen}
\usepackage{graphics}
\usepackage{subfigure}
\usepackage{mathrsfs}
\usepackage{amsbsy}
\usepackage{graphicx,epstopdf}
\usepackage{cite}
\usepackage{algorithm}
\usepackage{algorithmic}
\usepackage{dcolumn}
\usepackage{subeqnarray}
\usepackage{multirow}
\usepackage{bbm}

\newtheorem{theorem}{Theorem}
\newtheorem{lemma}{Lemma}
\newtheorem{remark}{Remark}
\newtheorem{corollary}{Corollary}
\newtheorem{definition}{Definition}
\newtheorem{claim}{Claim}
\newtheorem{proposition}{Proposition}
\newcommand{\hE}{\mathcal{E}}
\newcommand{\hN}{\mathcal{N}}
\newcommand{\hF}{\mathcal{F}}
\newcommand{\hP}{\mathcal{P}}
\newcommand{\hS}{\mathcal{S}}

\newcommand{\hL}{\mathcal{L}}

\newcommand{\ii}{\textbf{i}}
\newcommand{\re}{\mathrm{Re}}
\newcommand{\im}{\mathrm{Im}}

\newcommand{\diag}{\mathrm{diag}}
\newcommand{\eqdef}{:=}
\newcommand{\eps}{\varepsilon}

\newboolean{showcomments}
\setboolean{showcomments}{true}
\newcommand{\lgan}[1]{\ifthenelse{\boolean{showcomments}}
{ \textcolor{red}{(Lingwen says: #1)} } {} }
\newcommand{\slow}[1]{\ifthenelse{\boolean{showcomments}}
{ \textcolor{blue}{(Steven says:  #1)}}{}}
\newcommand{\addcite}[0]{\ifthenelse{\boolean{shocwomments}}
{ \textcolor{red}{(addcite)}}{}}
\newcommand{\addcites}[0]{\ifthenelse{\boolean{showcomments}}
{ \textcolor{red}{(addcite(s))}}{}}
\newcommand{\addref}[0]{\ifthenelse{\boolean{showcomments}}
{ \textcolor{Blue}{(addref)}}{}}
\newcommand{\todo}[1]{\ifthenelse{\boolean{showcomments}}
{ \textcolor{red}{(To do: #1)}} {} }

\begin{document}

\title{Exact Convex Relaxation of Optimal Power Flow in Radial Networks}
\author{Lingwen Gan, Na Li, Ufuk Topcu, and Steven H. Low
        \thanks{This work was supported by NSF NetSE grant CNS 0911041, ARPA-E grant DE-AR0000226, Southern California Edison, National Science Council of Taiwan, R.O.C, grant NSC 103-3113-P-008-001, Resnick Institute, and AFOSR award number FA9550-12-1-0302.

         Lingwen Gan, Na Li, and Steven H. Low are with the Engineering and Applied Science Department, California Institute of Technology, Pasadena, CA 91125 USA (e-mail: lgan@caltech.edu).

         Ufuk Topcu is with the Electrical and Systems Engineering Division, University of Pennsylvania, Philadelphia, PA 19104 USA.}
}
\maketitle

\begin{abstract}
The optimal power flow (OPF) problem determines 
power generation/demand that minimize a certain objective such as generation cost or power loss. 
It is nonconvex. We prove that, for radial networks, 
 after shrinking its feasible set slightly,
 the global optimum of OPF can be recovered via a second-order cone programming (SOCP) relaxation under a condition that can be checked a priori. The condition holds for the IEEE 13-, 34-, 37-, 123-bus networks and two real-world networks, and has a physical interpretation.\end{abstract}
\section{Introduction}
The optimal power flow (OPF) problem determines power generations/demands to minimize a certain 
objective such as generation cost or power loss. It has been one of the fundamental problems in power system operation since it was proposed in 1962 \cite{OPF}.
The OPF problem is increasingly important for distribution networks due to the advent of distributed generation (e.g., rooftop photovoltaic panels) and controllable loads (e.g., electric vehicles). Distributed generation is difficult to predict, calling the traditional ``generation follows demand'' strategy into question. Meanwhile, controllable loads provide significant potential to compensate for the randomness in distributed generation. To achieve this, solving the OPF problem in real-time is inevitable. Distribution networks are usually radial (with a tree topology).

The OPF problem is difficult to solve due to the nonconvex physical laws that goven power flow, and there are in general three ways to deal with this challenge: (i) linearize the power flow laws; (ii) look for local optima of the OPF problem; and (iii) convexify power flow laws, which are described in turn.

The power flow laws can be approximated by linear equations known as the DC power flow model \cite{Stott74,Alsac90,Stott09}, if 1) line resistances are small; 2) voltages are near their nominal values; and 3) voltage angle differences between adjacent buses are small. With DC power flow model, the OPF problem reduces to a linear program. This method is widely used in practice for transmission networks and often quite effective, but does not apply to distribution networks where line resistances are high and voltages deviate significantly from their nominal values. This method also does not apply to problems where reactive power flow or voltage deviations need to be optimized explicitly, e.g., power routing with FACTS devices \cite{xiao2002power} and Volt/VAR control \cite{Turitsyn10}.

Various algorithms have been developed to find local optima of the OPF problem, e.g., successive linear/quadratic programming \cite{Contaxis1986}, trust-region based methods \cite{Min2005,Sousa2011}, Lagrangian Newton method \cite{Baptista2005}, and interior-point methods \cite{Torres1998,Jabr2003,Capitanescu2007}. Some of these algorithms, especially the Newton-Ralphson based ones, are quite successful empirically, but in general, these algorithms are not guaranteed to convergence, nor converge to (nearly) optimal solutions.

There are two types of convex relaxations of the OPF problem: semidefinite programming (SDP) relaxations and second-order cone programming (SOCP) relaxations. It is proposed in \cite{Bai08,Javad12} to transform the nonconvex power flow constraints into linear constraints on a rank-one positive semidefinite matrix, and then remove the rank-one constraint to obtain an SDP relaxation. If the solution of the SDP relaxation is feasible for the OPF problem, then a global optimum of the OPF problem can be recovered. The SDP relaxation is called {\it exact} in this case. Strikingly, the SDP relaxation is exact for the IEEE 14-, 30-, 57-, and 118-bus networks \cite{Javad12}, and a more recent study on the computational speed and exactness of the SDP relaxation can be found in \cite{molzahnimplementation}.
Different SOCP relaxations have been proposed for different models, first in \cite{Jabr06}
for a branch flow model in polar coordinate, then in \cite{Masoud11, Farivar-2013-BFM-TPS}
for a branch flow model due to \cite{Baran89_capacitor_placement, Baran89_capacitor_sizing}, 
and in \cite{Sojoudi2012PES} for  a bus injection model.
In this paper we focus on the SOCP relaxation proposed in \cite{Masoud11,Farivar-2013-BFM-TPS}
and prove a sufficient condition for the relaxation to be exact.
For radial networks,  SOCP relaxation and  SDP relaxation are equivalent in the sense that there
is a bijection between their feasible sets
\cite{Bose-2012-BFMe-Allerton}.  Hence one should always solve SOCP relaxation instead of SDP
relaxation for radial networks since the former has a much lower computational complexity.

SDP/SOCP relaxations are in general not exact and counterexamples can be found in \cite{Lesieutre11}. 
Significant amount of work has been devoted to finding sufficient conditions under 
which these relaxations are exact for radial networks; see \cite{Low2013} for a
survey.
For \emph{AC radial} networks, these conditions roughly fall into three categories: 
\begin{enumerate}
\item The power injection constraints satisfy certain patterns  
\cite{Masoud11, Farivar-2013-BFM-TPS, Bose2011, Bose-2012-QCQPt, Zhang2013, Sojoudi2012PES, Sojoudi2013},
e.g., there are no lower bounds on the power injections (load over-satisfaction).
This sufficient condition, first proved in \cite{Bose-2012-QCQPt} and subsequently generalized 
in \cite{Sojoudi2013}, includes as special cases the load over-satisfaction condition in 
\cite{Masoud11, Farivar-2013-BFM-TPS, Bose2011, Sojoudi2012PES}
and in \cite[Theorem 7]{LavaeiTseZhang2012}, as well as the sufficient condition in 
\cite[Theorem 2]{Zhang2013}.

\item The phase angle difference across each line is bounded in terms of its $r/x$ ratio
\cite{Zhang2013, LavaeiTseZhang2012, lam2012distributed}.  When the voltage magnitude is fixed
this condition provides a nice geometric insight on why convex relaxations are exact.

\item The voltage upper bounds are relaxed plus some other conditions \cite{Gan12,Gan13}.
The main result in this paper generalizes and unifies this set of sufficient conditions;
see Section \ref{sec: prior works}.
\end{enumerate}

\subsection*{Summary of contributions}
The goal of this paper is to show that in radial networks, the SOCP relaxation is exact under a mild condition that can be checked a priori, after modifying the OPF problem. In particular, contributions of this paper are threefold.

First, we prove that {\it if optimal power injections lie in a region where voltage upper bounds do not bind, then the SOCP relaxation is exact under a mild condition}. The condition can be checked a priori and holds for the IEEE 13-, 34-, 37-, 123-bus networks and two real-world networks. The condition has a physical interpretation: it follows from the physical intuition that all upstream reverse power flows should increase if the power loss on a line is reduced.
Second, we {\it modify the OPF problem by imposing additional constraints on power injections}. The modification ensures the exactness of the SOCP relaxation under the aforementioned condition, while only eliminating feasible points that are close to voltage upper bounds. A modification is necessary to ensure an exact SOCP relaxation since otherwise examples exist where the SOCP relaxation is not exact.
Third, {\it this paper unifies and generalizes the results in \cite{Gan12,Gan13}}.

The rest of this paper is organized as follows. The OPF problem and the SOCP relaxation are introduced in Section \ref{sec: opf}. In Section \ref{sec: condition}, a sufficient condition that guarantees the exactness of the SOCP relaxation is provided. The condition consists of two parts: C1 and C2. C2 cannot be checked a priori, hence in Section \ref{sec: modification}, we propose a modified OPF problem whose corresponding SOCP is exact under C1. We compare C1 with prior works in Section \ref{sec: prior works} and present case studies in Section \ref{sec: case study}.
\section{The optimal power flow problem}\label{sec: opf}
This paper studies the optimal power flow (OPF) problem in distribution networks, which includes Volt/VAR control and demand response as special cases. In the following we present a model that incorporates nonlinear power flow and a variety of controllable devices including distributed generators, inverters, controllable loads, and shunt capacitors.

\subsection{Power flow model}
A distribution network is composed of buses and lines connecting these buses, and usually has a tree topology. The root of the tree is a {\it substation bus} that is connected to the transmission network. It has a fixed voltage and redistributes the bulk power it receives from the transmission network to other buses. Index the substation bus by 0 and the other buses by $1,\ldots,n$. Let $\hN\eqdef\{0,\ldots,n\}$ denote the collection of all buses and define $\hN^+:=\hN\backslash\{0\}$. Each line connects an ordered pair $(i,j)$ of buses where bus $j$ lies on the unique path from bus $i$ to bus 0. Let $\hE$ denote the collection of all lines, and abbreviate $(i,j)\in \hE$ by $i\rightarrow j$ whenever convenient.

For each bus $i\in \hN$, let $V_i$ denote its complex voltage and define $v_i\eqdef |V_i|^2$. Specifically the substation voltage $v_0$ is given and fixed. Let $s_i= p_i+\ii q_i$ denote the power injection of bus $i$ where $p_i$ and $q_i$ denote the real and reactive power injections respectively. Let $\hP_i$ denote the path (a collection of buses in $\hN$ and lines in $\hE$) from bus $i$ to bus 0. For each line $(i,j)\in\hE$, let $z_{ij}= r_{ij}+\ii x_{ij}$ denote its impedance. Let $I_{ij}$ denote the complex current from bus $i$ to bus $j$ and define $\ell_{ij}\eqdef |I_{ij}|^2$. Let $S_{ij}=P_{ij}+\ii Q_{ij}$ denote the sending-end power flow from bus $i$ to bus $j$ where $P_{ij}$ and $Q_{ij}$ denote the real and reactive power flow respectively.
Some of the notations are summarized in Fig. \ref{fig: notation}. We use a letter without subscripts to denote a vector of the corresponding quantities, e.g., $v=(v_i)_{i\in\hN^+}$, $\ell=(\ell_{ij})_{(i,j)\in\hE}$. Note that subscript 0 is not included in nodal quantities such as $v$ and $s$. For a complex number $a\in\mathbb{C}$, let $\bar{a}$ denote the conjugate of $a$.
	\begin{figure}[h]
     	\centering
     	\includegraphics[scale=0.4]{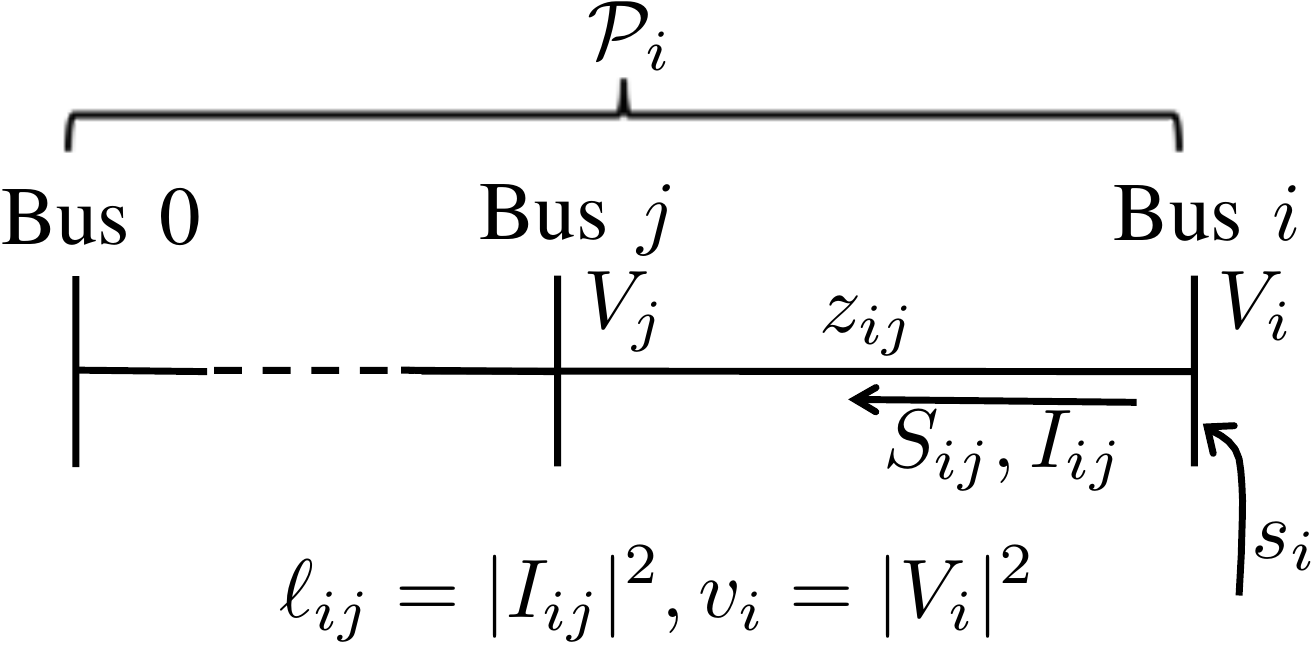}
      	\caption{Some of the notations.}
      	\label{fig: notation}
	\end{figure}

Given the network graph $(\hN, \hE)$, the impedance $z$, and the substation voltage $v_0$, then the other variables $(s,S,v,\ell,s_0)$ are described by the {\it branch flow model}:
\begin{subequations}\label{PF}
\begin{align}
& S_{ij} = s_i + \sum_{h:\,h\rightarrow i} (S_{hi}-z_{hi}\ell_{hi}), & \forall (i,j)\in\hE;\label{BFM S}\\
& 0 = s_0 + \sum_{h:\,h\rightarrow 0} (S_{h0}-z_{h0}\ell_{h0}); & \label{BFM s}\\
& v_i-v_j = 2\re(\bar{z}_{ij}S_{ij})-|z_{ij}|^2\ell_{ij}, &\forall (i,j)\in\hE;\label{BFM v}\\
& \ell_{ij} = \frac{|S_{ij}|^2}{v_i}, &\forall (i,j)\in\hE\label{BFM ell}
\end{align}
\end{subequations}
for radial networks \cite{Baran89_capacitor_sizing,Baran89_network_reconfiguration}.

\subsection{The OPF problem}
\label{sec: control elements}
We consider the following controllable devices in a distribution network: distributed generators, inverters, controllable loads such as electric vehicles and smart appliances, and shunt capacitors. Real and reactive power generation/consumption of these devices can be controlled to achieve certain objectives. For example, in Volt/VAR control, reactive power injection of inverters and shunt capacitors are controlled to regulate voltages; in demand response, real power consumption of controllable loads is reduced or shifted in response to power supply conditions. Mathematically, power injection $s$ is the control variable, after specifying which the other variables $(S,v,\ell,s_0)$ are determined by the power flow laws in \eqref{PF}.

The power injection $s_i$ of a bus $i\in \hN^+$ is constrained to be in an pre-specified set $\mathcal{S}_i$, i.e.,
	\begin{equation}\label{constraint s}
	s_i \in \mathcal{S}_i, \qquad i\in \hN^+.
	\end{equation}
The set $\mathcal{S}_i$ for some controllable devices are:
\begin{itemize}
\item If $s_i$ represents a shunt capacitor with nameplate capacity $\overline{q}_i$, then
	$\mathcal{S}_i = \{s\in\mathbb{C}~|~\re(s)=0, ~\im(s)=0\text{ or }\overline{q}_i\}.$
Note that $\mathcal{S}_i$ is nonconvex and disconnected in this case.
\item If $s_i$ represents a solar panel with generation capacity $\overline{p}_i$, that is connected to the grid through an inverter with nameplate capacity $\overline{s}_i$, then
	$\mathcal{S}_i = \{s\in\mathbb{C}~|~0\leq \re(s)\leq \overline{p}_i, ~|s|\leq \overline{s}_i \}.$
\item If $s_i$ represents a controllable load with constant power factor $\eta$, whose real power consumption can vary continuously from $-\overline{p}_i$ to $-\underline{p}_i$ (here $\underline{p}_i\leq\overline{p}_i\leq0$), then
	$\mathcal{S}_i = \left\{s\in\mathbb{C} ~|~
    	\underline{p}_i\leq \re(s)\leq \overline{p}_i, ~\im(s)=\sqrt{1-\eta^2}\re(s)/\eta
    	\right\}.$
\end{itemize}
Note that $s_i$ can represent the aggregate power injection of multiple such devices with an appropriate $\mathcal{S}_i$, and that the set $\mathcal{S}_i$ is not necessarily convex or connected.

An important goal of control is to regulate the voltages within a range. This is captured by pre-specified voltage lower and upper bounds $\underline{v}_i$ and  $\overline{v}_i$ (in per unit value), i.e.,
    	\begin{equation}\label{constraint v}
	\underline{v}_i \leq v_i \leq \overline{v}_i, \qquad i\in \hN^+.
	\end{equation}
For example, if 5\% voltage deviation from nominal values is allowed, then $0.95^2\leq v_i\leq 1.05^2$. We consider the control objective
	\begin{equation}\label{objective}
	C(s,s_0) = \sum_{i\in\hN} f_i(\re(s_i))
	\end{equation}
where $f_i:\mathbb{R}\rightarrow\mathbb{R}$ denotes the generation cost at bus $i$ for $i\in\hN$. If $f_i(x)=x$ for $i\in\hN$, then $C$ is the total power loss in the network.

The OPF problem seeks to minimize the generation cost \eqref{objective}, subject to power flow constraint \eqref{PF}, power injection constraint \eqref{constraint s}, and voltage constraint \eqref{constraint v}:
	\begin{subequations}
    	\begin{align}
    	\textbf{OPF:}~\min~~ & \sum_{i\in\hN} f_i(\re(s_i)) \nonumber\\
	\mathrm{over}~~ & s,S,v,\ell,s_0 \nonumber\\
	\mathrm{s.t.}~~ & S_{ij} = s_i + \sum_{h:\,h\rightarrow i} (S_{hi}-z_{hi}\ell_{hi}), \quad \forall (i,j)\in\hE; \label{OPF S}\\
	& 0 = s_0 + \sum_{h:\,h\rightarrow 0} (S_{h0}-z_{h0}\ell_{h0}); \label{OPF s}\\
	& v_i-v_j = 2\re(\bar{z}_{ij}S_{ij})-|z_{ij}|^2\ell_{ij}, \quad \forall (i,j)\in\hE; \label{OPF v}\\
	& \ell_{ij} = \frac{|S_{ij}|^2}{v_i}, \quad \forall (i,j)\in\hE; \label{OPF ell}\\
	& s_i \in \mathcal{S}_i, \quad i\in \hN^+; \label{OPF constraint s}\\
	& \underline{v}_i \leq v_i \leq \overline{v}_i, \quad i\in \hN^+.\label{OPF constraint v}
    	\end{align}
	\end{subequations}
The following assumptions are made on OPF throughout this work.
\begin{itemize}
\item[A1] The network $(\hN,\hE)$ is a tree. Distribution networks are usually radial networks.
\item[A2] The substation voltage $v_0$ is fixed and given. In practice, $v_0$ can be modified several times a day, and therefore can be considered as a given constant at the minutes timescale of OPF.
\item[A3] Line resistances and reactances are strictly positive, i.e., $r_{ij}>0$ and $x_{ij}>0$ for $(i,j)\in\hE$. In practice, $r_{ij}>0$ since lines are passive (consume power), and $x_{ij}>0$ since lines are inductive.
\item[A4] Voltage lower bounds are strictly positive, i.e., $\underline{v}_i>0$ for $i\in\hN^+$. In practice, $\underline{v}_i$ is slightly below 1p.u..
\end{itemize}

The equality constraint \eqref{OPF ell} is nonconvex, and one can relax it to inequality constraints to obtain the following second-order cone programming (SOCP) relaxation \cite{Masoud11,Farivar-2013-BFM-TPS}:
    	\begin{align}
    	\textbf{SOCP:}~\min~~ & \sum_{i\in\hN} f_i(\re(s_i)) \nonumber\\
	\mathrm{over}~~ & s,S,v,\ell,s_0 \nonumber\\
	\mathrm{s.t.}~~ & \eqref{OPF S}-\eqref{OPF v}, ~\eqref{OPF constraint s}-\eqref{OPF constraint v};\nonumber\\
	& \ell_{ij} \geq \frac{|S_{ij}|^2}{v_i}, \quad \forall (i,j)\in\hE.\label{relax}
    	\end{align}
Note that SOCP is not necessarily convex, since we allow $f_i$ to be nonconvex for some $i\in\hN$ and $\mathcal{S}_i$ to be nonconvex for some $i\in\hN^+$. Nonetheless, we call it SOCP for brevity.

If an SOCP solution $w=(s,S,v,\ell,s_0)$ is feasible for OPF, i.e., $w$ satisfies \eqref{OPF ell}, then $w$ is a global optimum of OPF. This motivates the following definition of {\it exactness} for SOCP.
	\begin{definition}\label{def: exact}
	SOCP is {\it exact} if every of its solutions satisfies \eqref{OPF ell}.
	\end{definition}
\section{A sufficient condition}\label{sec: condition}
We provide a sufficient condition that ensures SOCP to be exact in this section. This condition is composed of two parts: C1 and C2. C1 is a mild condition that only depends on SOCP parameters. It follows from the physical intuition that all upstream reverse power flows should increase if the power loss on a line is reduced. C2 depends on SOCP solutions and cannot be checked a priori, but motivates us to modify OPF such that the corresponding SOCP is exact under C1. The modified OPF problem will be discussed in Section \ref{sec: modification}.

\subsection{Statement of the condition}\label{sec: statement}
We start with introducing the notations that will be used in the statement of the condition. One can ignore the $\ell$ terms in \eqref{BFM S} and \eqref{BFM v} to obtain the {\it Linear DistFlow Model} \cite{Baran89_capacitor_sizing, Baran89_network_reconfiguration}
	\begin{eqnarray*}
	S_{ij} = s_i + \sum_{h:\,h\rightarrow i}S_{hi}, && \forall(i,j)\in\hE;\\
	v_i - v_j = 2\re(\bar{z}_{ij}S_{ij}), && \forall(i,j)\in\hE.
	\end{eqnarray*}
Let $(\hat{S},\hat{v})$ denote the solution of the Linear DistFlow model, then
	\begin{eqnarray*}
	\hat{S}_{ij}(s) = \sum_{h:\,i\in \hP_h} s_h, && \forall(i,j)\in\hE;\\
	\hat{v}_i(s) \eqdef v_0 + 2\sum_{(j,k)\in \hP_i}\re\left(\bar{z}_{jk}\hat{S}_{jk}( s)\right), && \forall i\in\hN
	\end{eqnarray*}
as in Fig. \ref{fig: approximation}. Physically, $\hat{S}_{ij}(s)$ denote the sum of power injections $s_h$ towards bus 0 that go through line $(i,j)$. Note that $(\hat{S}(s),\hat{v}(s))$ is affine in $s$, and equals $(S,v)$ if and only if line loss $z_{ij}\ell_{ij}$ is 0 for $(i,j)\in\hE$.
	\begin{figure}[!ht]
     	\centering
     	\includegraphics[scale=0.4]{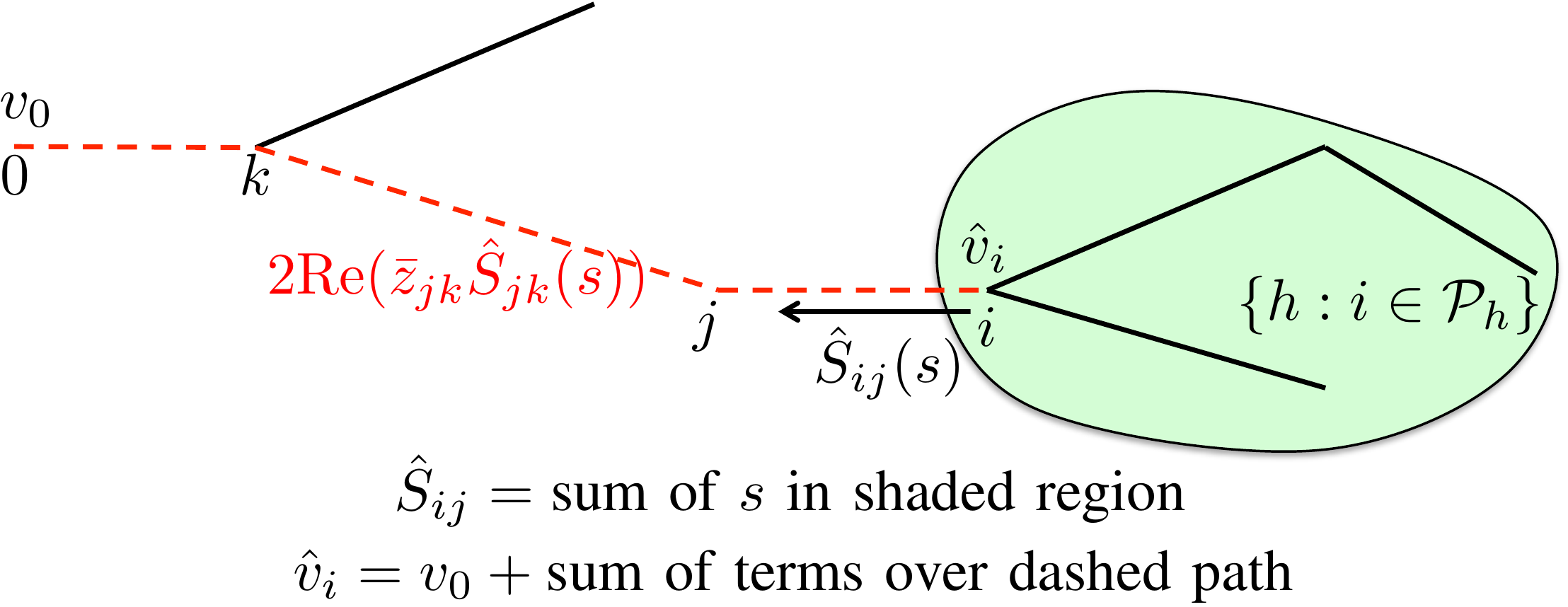}
      	\caption{Illustration of $\hat{S}_{ij}$ and $\hat{v}_i$. The shaded region is downstream of bus $i$, and contains the buses $\{h:i\in\hP_h\}$. Quantity $\hat{S}_{ij}(s)$ is defined to be the sum of bus injections $s$ in the shaded region. The dashed lines constitute the path $\hP_i$ from bus $i$ to bus 0. Quantity $\hat{v}_i(s)$ is defined as $v_0$ plus the terms $2\re(\bar{z}_{jk}\hat{S}_{jk}(s))$ over the dashed path.}
      	\label{fig: approximation}
	\end{figure}	
For two complex numbers $a,b\in\mathbb{C}$, let $a\leq b$ denote $\re(a)\leq \re(b)$ and $\im(a)\leq \im(b)$. For two vectors $a,b$ of the same dimension, $a\leq b$ denotes componentwise inequality. Define $<$, $>$, and $\geq$ similarly.
	\begin{lemma}\label{lemma: v}
	If $(s,S,v,\ell,s_0)$ satisfies \eqref{BFM S}--\eqref{BFM v} and $\ell\geq0$ componentwise, then $S\leq \hat{S}(s )$ and $v\leq \hat{v}(s)$.
	\end{lemma}
Lemma \ref{lemma: v} implies that $\hat{v}(s)$ and $\hat{S}(s)$ provide upper bounds on $v$ and $S$. The lemma is proved in Appendix \ref{app: lemma v}. Let $\hat{P}(s)$ and $\hat{Q}(s)$ denote the real and imaginary part of $\hat{S}(s)$ respectively. Then
	\begin{eqnarray*}
	\hat{P}_{ij}(s=p+\ii q) = \hat{P}_{ij}({p}) = \sum_{h:\,i\in\hP_h} p_h, && (i,j)\in\hE;\\
	\hat{Q}_{ij}(s=p+\ii q) = \hat{Q}_{ij}({q}) = \sum_{h:\,i\in\hP_h} q_h, && (i,j)\in\hE.
	\end{eqnarray*}
Assume that there exists $\overline{p}_i$ and $\overline{q}_i$ such that
    	\begin{equation*}\label{relax constraints}
	\mathcal{S}_i\subseteq\{s\in\mathbb{C}~|~\re(s)\leq\overline{p}_i,~\im(s)\leq\overline{q}_i\}
	\end{equation*}
for $i\in\hN^+$ as in Fig. \ref{fig: constraint}, i.e., $\re(s_i)$ and $\im(s_i)$ are upper bounded by $\overline{p}_i$ and $\overline{q}_i$ respectively.
	\begin{figure}[!ht]
     	\centering
     	\includegraphics[scale=0.4]{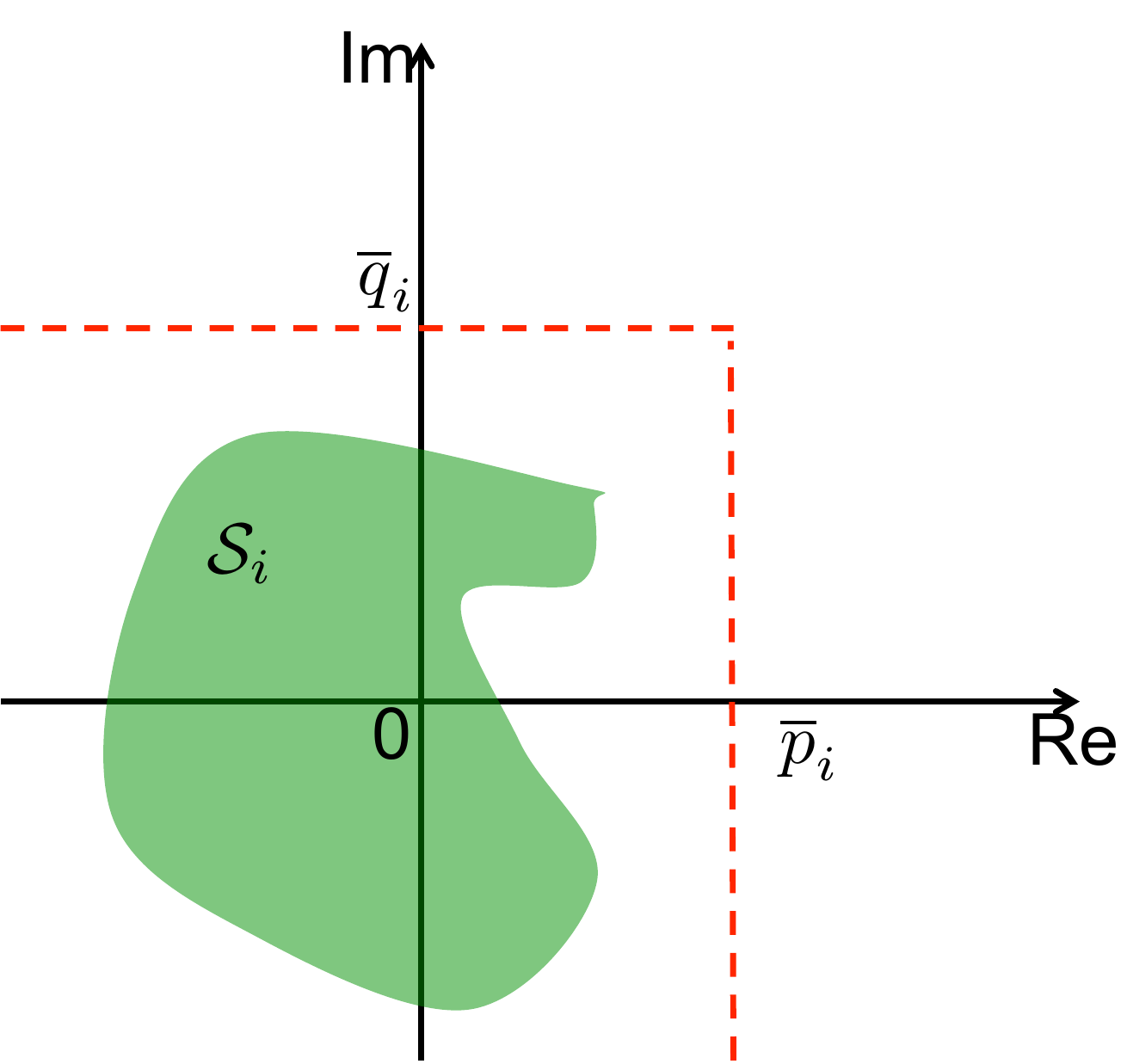}
      	\caption{We assume that $\mathcal{S}_i$ lies in the left bottom corner of $(\overline{p}_i,\overline{q}_i)$, but do not assume that $\mathcal{S}_i$ is convex or connected.}
      	\label{fig: constraint}
	\end{figure}
Note that we do {\it not} assume $\mathcal{S}_i$ to be convex or connected. Define $a^+\eqdef\max\{a,0\}$ for $a\in\mathbb{R}$, let $I:=\diag(1,1)$ denote the $2\times2$ identity matrix, and define
    $$u_i:=u_{ij}:=\begin{pmatrix} r_{ij} \\ x_{ij} \end{pmatrix}, \qquad
    \underline{A}_i := \underline{A}_{ij} := I - \frac{2}{\underline{v}_i} \begin{pmatrix} r_{ij} \\ x_{ij} \end{pmatrix} \left(\hat{P}^+_{ij}(\overline{p}) ~~\hat{Q}_{ij}^+(\overline{q})\right)    $$
for $(i,j)\in\hE$. Since $(i,j_1)\in\hE$ and $(i,j_2)\in\hE$ implies $j_1=j_2$, $\underline{A}_i$ and $u_i$ are well-defined for $i\in\hN^+$.
	
Further, let $\hL\eqdef\{l\in\hN~|~\nexists k\in\hN \text{ such that }k\rightarrow l\}$ denote the collection of leaf buses in the network. For a leaf bus $l\in\hL$, let $n_l+1$ denote the number of buses on path $\hP_l$, and suppose
    	$$\hP_l=\left\{l_{n_l}\rightarrow l_{n_l-1}\rightarrow \ldots\rightarrow l_1\rightarrow l_0\right\}$$
with $l_{n_l}=l$ and $l_0=0$ as in Fig. \ref{fig: path}.
	\begin{figure}[h]
     	\centering
     	\includegraphics[scale=0.4]{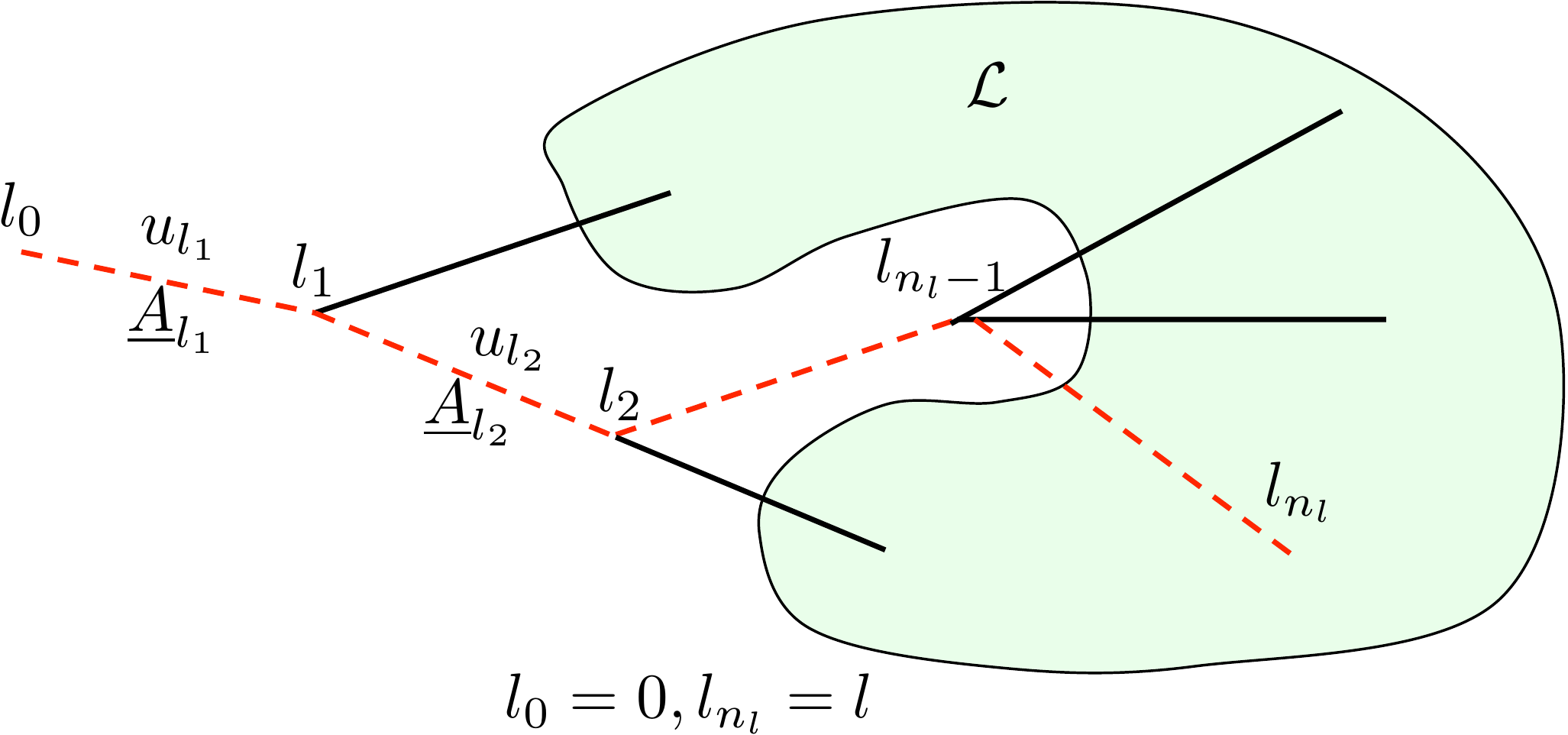}
      	\caption{The shaded region denotes the collection $\hL$ of leaf buses, and the path $\hP_l$ of a leaf bus $l\in\hL$ is illustrated by a dashed line.}
      	\label{fig: path}
	\end{figure}
Let
	$$\hS_{\mathrm{volt}} := \{s\in\mathbb{C}^n ~|~ \hat{v}_i(s)\leq \overline{v}_i \text{ for }i\in\hN^+\}$$
denote the power injection region where $\hat{v}(s)$ is upper bounded by $\overline{v}$. Since $v\leq\hat{v}(s)$ (Lemma \ref{lemma: v}), the set $\hS_{\mathrm{volt}}$ is a power injection region where voltage upper bounds do not bind.
	
The following theorem provides a sufficient condition that guarantees the exactness of SOCP.
    \begin{theorem}\label{thm: condition}
    Assume that $f_0$ is strictly increasing, and that there exists $\overline{p}_i$ and $\overline{q}_i$ such that $\mathcal{S}_i\subseteq\{s\in\mathbb{C}~|~\re(s)\leq\overline{p}_i,~\im(s)\leq\overline{q}_i\}$ for $i\in\hN^+$. Then SOCP is exact if the following conditions hold:
    \begin{itemize}
    \item[C1] $\underline{A}_{l_s}\underline{A}_{l_{s+1}} \cdots \underline{A}_{l_{t-1}}u_{l_t}>0$ for any $l\in\hL$ and any $s,t$ such that $1\leq s\leq t\leq n_l$;
    \item[C2] every SOCP solution $w=(s, S,v,\ell, s_0)$ satisfies $s\in\hS_{\mathrm{volt}}$.
    \end{itemize}
    \end{theorem}
    
Theorem \ref{thm: condition} implies that if C2 holds, i.e., optimal power injections lie in the region $\hS_{\mathrm{volt}}$ where voltage upper bounds do not bind, then SOCP is exact under C1. The theorem is proved in Appendix \ref{app: condition}. C2 depends on SOCP solutions and cannot be checked a priori. This drawback motivates us to modify OPF such that the corresponding SOCP is exact under C1, as will be discussed in Section \ref{sec: modification}.

\subsection{Interpretation of C1}\label{sec: interpretation}
We illustrate C1 through a linear network as in Fig. \ref{fig: linear C1}. The collection of leaf buses is a singleton $\hL=\{n\}$, and the path from the only leaf bus $n$ to bus 0 is $\hP_n=\{n\rightarrow n-1\rightarrow\cdots\rightarrow 1\rightarrow 0\}$. Then, C1 takes the form
	$$\underline{A}_{s}\underline{A}_{s+1}\cdots\underline{A}_{t-1}u_t>0,\qquad 1\leq s\leq t\leq n.$$
That is, given any network segment $(s-1,t)$ where $1\leq s\leq t\leq n$, the multiplication $\underline{A}_s\underline{A}_{s+1}\cdots\underline{A}_{t-1}$ of $\underline{A}$ over the segment $(s-1,t-1)$ times $u_t$ is strictly positive.
	\begin{figure}[!h]
     	\centering
     	\includegraphics[scale=0.4]{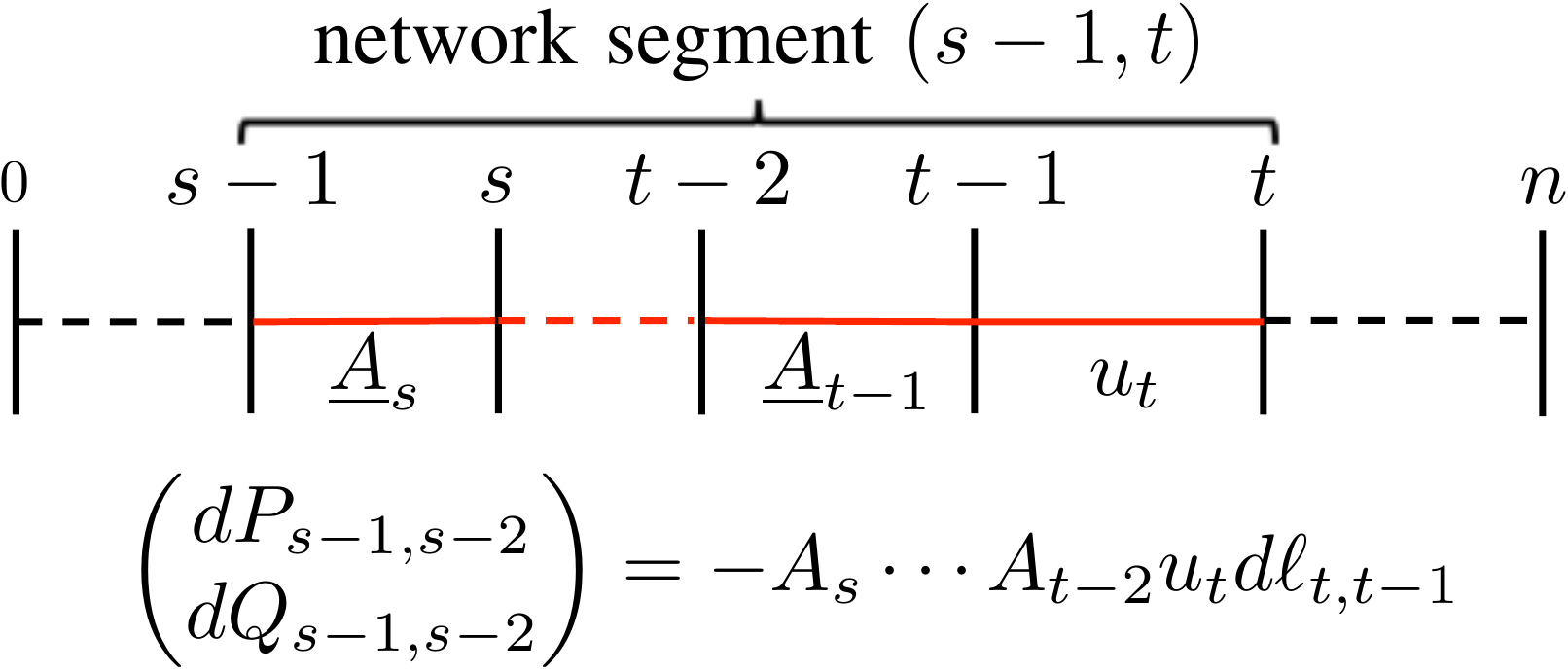}
      	\caption{In the above linear network, $\hL=\{n\}$ and $\hP_n=\{n\rightarrow n-1\rightarrow \cdots \rightarrow 1\rightarrow 0\}$. C1 requires that given any highlighted segment $(s-1,t)$ where $1\leq s\leq t\leq n$, the multiplication of $\underline{A}$ over $(s-1,t-1)$ times $u_t$ is strictly positive (componentwise).}
      	\label{fig: linear C1}
	\end{figure}

C1 only depends on SOCP parameters $(r,x,\overline{p},\overline{q},\underline{v})$ and therefore can be checked a priori. Furthermore, C1 can be checked efficiently since $\underline{A}$ and $u$ are simple functions of $(r,x,\overline{p},\overline{q},\underline{v})$ that can be computed in $O(n)$ time, and there are no more than $n(n+1)/2$ inequalities in C1.

\begin{proposition}\label{prop: smaller injections}
If $(\overline{p},\overline{q}) \leq (\overline{p}',\overline{q}')$ and C1 holds for $(r,x,\overline{p}',\overline{q}',\underline{v})$, then C1 also holds for $(r,x,\overline{p},\overline{q},\underline{v})$.
\end{proposition}
Proposition \ref{prop: smaller injections} implies that the smaller power injections, the more likely C1 holds. It is proved in Appendix \ref{app: smaller injections}.

\begin{proposition}\label{pro: no reverse}
If $(\overline{p},\overline{q})\leq0$, then C1 holds.
\end{proposition}
Proposition \ref{pro: no reverse} implies that if every bus only consumes real and reactive power, then C1 holds. This is because when $(\overline{p},\overline{q})\leq0$, the quantities $\hat{P}_{ij}(\overline{p})\leq0$, $\hat{Q}_{ij}(\overline{q})\leq0$ for $(i,j)\in\hE$. It follows that $\underline{A}_i=I$ for $i\in\hN^+$. Hence, $\underline{A}_{l_s}\cdots \underline{A}_{l_{t-1}}u_{l_t} = u_{l_t}>0$ for any $l\in\hL$ and any $s,t$ such that $1\leq s\leq t\leq n_l$.

For practical parameter ranges of $(r,x,\overline{p},\overline{q},\underline{v})$, line resistance and reactance $r_{ij},x_{ij}\ll1$ for $(i,j)\in\hE$, line flow $\hat{P}_{ij}(\overline{p}),\hat{Q}_{ij}(\overline{q})=O(1)$ for $(i,j)\in\hE$, and voltage lower bound $\underline{v}_i\approx1$ for $i\in\hN^+$. Hence, $\underline{A}_i$ is close to $I$ for $i\in\hN^+$, and therefore C1 is likely to hold. As will be seen in the numeric studies in Section \ref{sec: case study}, C1 holds for several test networks, including those with big $(\overline{p},\overline{q})$ (high penetration of distributed generation).

C1 has a physical interpretation. Recall that $S_{k,k-1}$ denotes the reverse power flow on line $(k,k-1)$ for $k=1,\ldots,n$ and introduce $S_{0,-1}:=-s_0$ for convenience. If the power loss on a line is reduced, then all upstream reverse power flows seem nature to increase. More specifically, the power loss on line $(t,t-1)$ where $t\in\{1,2,\ldots,n\}$ can be reduced by decreasing the current $\ell_{t,t-1}$ by $d\ell_{t,t-1}<0$, and physical intuition tells us that reverse power flow $S_{s-1,s-2}$ is likely to increase, i.e., $dS_{s-1,s-2}>0$, for $s=1,2,\ldots,t$. Now assume that indeed $dS_{s-1,s-2} = dP_{s-1,s-2} + \ii dQ_{s-1,s-2}>0$ for $s=1,\ldots,t$. It can be verified that
	$ (dP_{t-1,t-2} ~~dQ_{t-1,t-2})^T = -u_t d\ell_{t,t-1}, $
and one can compute from \eqref{PF} the Jacobian matrix
	$$A_k:=\begin{pmatrix}
	\frac{\partial P_{k-1,k-2} } {\partial P_{k,k-1} } & \frac{\partial P_{k-1,k-2} } {\partial Q_{k,k-1} } \\
	\frac{\partial Q_{k-1,k-2} } {\partial P_{k,k-1} } & \frac{\partial Q_{k-1,k-2} } {\partial Q_{k,k-1} }
	\end{pmatrix}
	= I - \frac{2}{v_i}\begin{pmatrix} r_{k,k-1} \\ x_{k,k-1} \end{pmatrix}
	(P_{k,k-1} ~~Q_{k,k-1})$$
for $k=1,\ldots,n$. Therefore reverse power flow $S_{s-1,s-2}$ changes by $dS_{s-1,s-2}=dP_{s-1,s-2}+\ii dQ_{s-1,s-2}$ where
	$$(dP_{s-1,s-2} ~~dQ_{s-1,s-2})^T = -A_sA_{s+1}\cdots A_{t-1}u_t d\ell_{t,t-1},$$
according to the chain rule, for $s=1,\ldots,t$. Then, $dS_{s-1,s-2}>0$ implies
	\begin{equation}\label{interpretation}
	A_sA_{s+1}\cdots A_{t-1}u_t >0
	\end{equation}
for $s=1,2,\ldots,t$. Note that $\underline{A}_k$ is obtained by replacing $(P,Q,v)$ in $A_k$ by $(\hat{P}^+(\overline{p}),\hat{Q}^+(\overline{q}),\underline{v})$ (so that $\underline{A}_k$ only depends on SOCP parameters), and then \eqref{interpretation} becomes C1.

\subsection{Proof idea}\label{sec: idea}
We present the proof idea of Theorem \ref{thm: condition} via a 3-bus linear network as in Fig. \ref{fig: 3-bus},
	\begin{figure}[!h]
     	\centering
     	\includegraphics[scale=0.4]{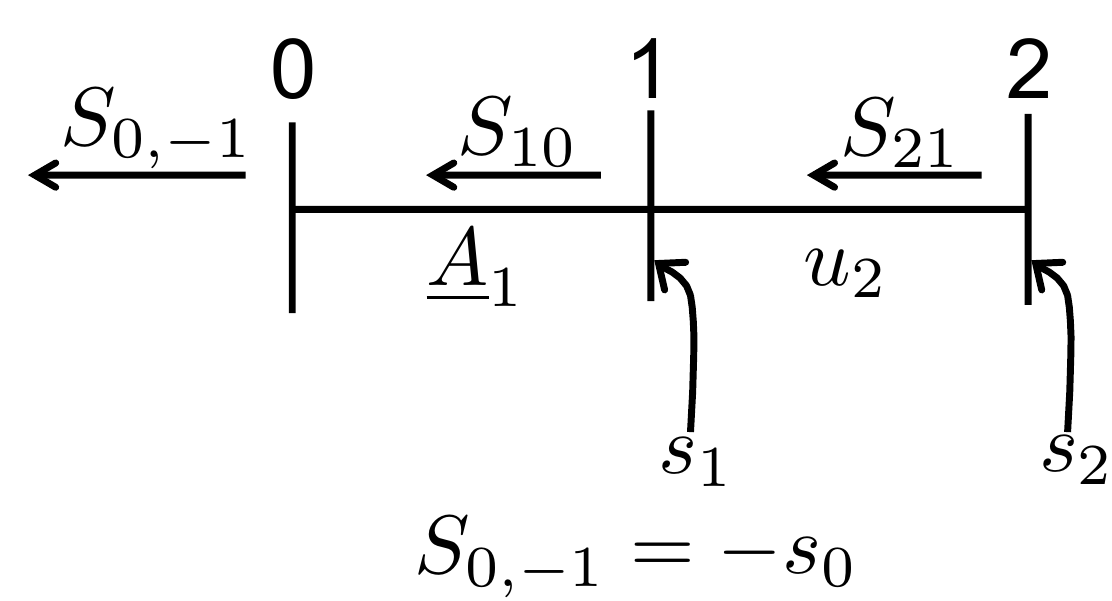}
      	\caption{A 3-bus linear network.}
      	\label{fig: 3-bus}
	\end{figure}
and the proof for general tree networks is provided in Appendix \ref{app: condition}. Assume that $f_0$ is strictly increasing, and that C1 and C2 hold. If SOCP is not exact, then there exists an SOCP solution $w=(s,S,v,\ell,s_0)$ that violates \eqref{OPF ell}. We are going to construct another feasible point $w'=(s',S',v',\ell',s_0')$ of SOCP that has a smaller objective value than $w$. This contradicts with $w$ being optimal, and therefore SOCP is exact.

The construction of $w'$ is as follows. There are two ways \eqref{OPF ell} gets violated: 1) violated on line $(1,0)$; 2) satisfied on line $(1,0)$ but violated on line $(2,1)$. To illustrate the proof idea, we focus on the second case, i.e., the case where $\ell_{10}=|S_{10}|^2/v_1$ and $\ell_{21}>|S_{21}|^2/v_2$. In this case, the construction of $w'$ is
	\begin{subequations}\label{construction}
	\begin{align}
	\textbf{Initialization:} \qquad
	& s'=s; \label{construct s}\\
	& S_{21}'=S_{21}; \label{construct S21}\\
	\textbf{Forward sweep:} \qquad
	& \ell_{21}'=|S_{21}'|^2/v_2; \label{construct ell21}\\
	& S_{10}'=S_{21}'-z_{21}\ell_{21}'+s_1';\label{construct S10}\\
	& \ell_{10}'=|S_{10}'|^2/v_1; \label{construct ell10}\\
	& S_{0,-1}' = S_{10}'-z_{10}\ell_{10}'; \label{construct S0}\\
	\textbf{Backward sweep:} \qquad
	& v_1' = v_0 + 2\re(\bar{z}_{10}S_{10}') - |z_{10}|^2\ell_{10}'; \label{construct v1}\\
	& v_2' = v_1' + 2\re(\bar{z}_{21}S_{21}') - |z_{21}|^2\ell_{21}' \label{construct v2}
	\end{align}
	\end{subequations}
where $s_0'=-S_{0,-1}'$. The construction consists of three steps:
\begin{itemize}
\item[S1] In the initialization step, $s'$ and $S_{21}'$ are initialized as the corresponding values in $w$. Hence, $w'$ satisfies \eqref{OPF constraint s}.
\item[S2] In the forward sweep step, $\ell_{k,k-1}'$ and $S_{k-1,k-2}'$ are recursively constructed for $k=2,1$ by alternatively applying \eqref{OPF ell} (with $v'$ replaced by $v$) and \eqref{OPF S}/\eqref{OPF s}. This recursive construction updates $\ell'$ and $S'$ alternatively along the path $\hP_2$ from bus 2 to bus 0, and is therefore called a {\it forward sweep}. It is clear that $w'$ satisfies \eqref{OPF S} and \eqref{OPF s}. Besides, $w'$ satisfies \eqref{relax} if and only if $v'\geq v$.
\item[S3] In the backward sweep step, $v_k'$ is recursively constructed for $k=1,2$ by applying \eqref{OPF v}. This recursive construction updates $v'$ along the negative direction of $\hP_2$ from bus 0 to bus 2, and is therefore called a {\it backward sweep}. It is clear that $w'$ satisfies \eqref{OPF v}.
\end{itemize}

We will show that $w'$ is feasible for SOCP and has a smaller objective value than $w$. This result follows from the following two claims.
\begin{claim}\label{claim C1}
$\text{C1} ~\Rightarrow~ S_{k,k-1}'>S_{k,k-1} \text{ for } k=0,1 ~\Rightarrow~ v'\geq v$.
\end{claim}
\begin{claim}\label{claim C2}
$\text{C2} ~\Rightarrow~ v'\leq\overline{v}$.
\end{claim}
It follows from Claims \ref{claim C1} and \ref{claim C2} that $\underline{v}\leq v\leq v'\leq\overline{v}$, and therefore $w'$ satisfies \eqref{OPF constraint v}. As discussed in Step S2, $w'$ also satisfies \eqref{relax} since $v'\geq v$. Hence, $w'$ is feasible for SOCP. The point $w'$ has a smaller objective value than $w$ because
	\begin{eqnarray*}
	\sum_{i\in\hN} f_i(\re(s_i')) - \sum_{i\in\hN} f_i(\re(s_i)) = f_0(-\re(S_{0,-1}')) - f_0(-\re(S_{0,-1})) < 0.
	\end{eqnarray*}
This contradicts the optimality of $w$, and therefore SOCP is exact. To complete the proof, we are left to prove Claims \ref{claim C1} and \ref{claim C2}.

\noindent{\it Proof of Claim \ref{claim C1}: }
First show that C1 implies $S_{k,k-1}'>S_{k,k-1}$ for $k=0,1$. Define $\Delta s:=s'-s$, $\Delta S:=S'-S$, $\Delta \ell:=\ell'-\ell$, and $\Delta v:=v'-v$. It is assumed that $\ell_{21}>|S_{21}|^2/v_2$, and therefore $\ell_{21}'=|S_{21}'|^2/v_2=|S_{21}|^2/v_2<\ell_{21}$, i.e., $\Delta \ell_{21}<0$. It follows that
	$$\Delta S_{10} = \Delta S_{21}' - z_{21}\Delta\ell_{21}+\Delta s_1 = - z_{21}\Delta\ell_{21} > 0.$$
Recalling that $S=P+\ii Q$ and that $u_2=(r_{21} ~x_{21})^T$, one has $(\Delta P_{10} ~\Delta Q_{10})^T = -u_2\Delta \ell_{21}$. It follows from \eqref{construct ell10}--\eqref{construct S0} that $S_{0,-1}'=S_{10}'-z_{10}|S_{10}'|^2/v_1$. It is assumed that $\ell_{10}=|S_{10}|^2/v_1$, therefore $S_{0,-1}=S_{10}-z_{10}|S_{10}|^2/v_1$. Hence, $\Delta S_{0,-1}= \Delta S_{10}-z_{10}(|S_{10}'|^2-|S_{10}|^2)/v_1$, which can be written as
	\begin{equation}\label{delta S0}
	\begin{pmatrix} \Delta P_{0,-1} \\ \Delta Q_{0,-1} \end{pmatrix}
	 = B_1
	 \begin{pmatrix} \Delta P_{10} \\ \Delta Q_{10} \end{pmatrix}
	 = - B_1 u_{2}\Delta \ell_{21}
	\end{equation}
where
	$$B_1 = I - \frac{2}{v_1}\begin{pmatrix} r_{10} \\ x_{10} \end{pmatrix}
	\left(\frac{P_{10}+P_{10}'}{2} ~~\frac{Q_{10}+Q_{10}'}{2}\right).$$
	 
When C1 holds, one has $\underline{A}_1u_2>0$, and therefore
	\begin{eqnarray}
	B_1u_2 &=& \underline{A}_1u_2 + (B_1-\underline{A}_1)u_2 \nonumber\\
	&>& (B_1-\underline{A}_1)u_2 \nonumber\\
	&=& \begin{pmatrix} r_{10} \\ x_{10} \end{pmatrix}
	\left(
	\frac{2\hat{P}_{10}^+(\overline{p})}{\underline{v}_1}        -         \frac{P_{10}+P_{10}'}{v_1},        \quad
	\frac{2\hat{Q}_{10}^+(\overline{q})}{\underline{v}_1}        -         \frac{Q_{10}+Q_{10}'}{v_1}
	\right).\label{expansion}
	\end{eqnarray}
According to Lemma \ref{lemma: v}, one has $P_{10}\leq\hat{P}_{10}({p})\leq\hat{P}_{10}(\overline{p})\leq\hat{P}_{10}^+(\overline{p})$. Similarly, $P_{10}'\leq\hat{P}_{10}^+(\overline{p})$. Therefore
	$$\frac{2\hat{P}_{10}^+(\overline{p})}{\underline{v}_1}   \geq    
	\frac{2\hat{P}_{10}^+(\overline{p})}{v_1}     \geq
	\frac{P_{10}+P_{10}'}{v_1}.$$
Similarly, one has
	$$\frac{2\hat{Q}_{10}^+(\overline{q})}{\underline{v}_1}   \geq
	\frac{Q_{10}+Q_{10}'}{v_1}.$$
Then it follows from \eqref{expansion} that
	$B_1u_2>(B_1-\underline{A}_1)u_2\geq0.$
Then it follows from \eqref{delta S0} that $\Delta S_{0,-1}>0$. This completes the proof that C1 implies $\Delta S_{k,k-1}>0$ for $k=0,1$.
	
Next we show that $\Delta S_{k,k-1}>0$ for $k=0,1$ implies $v'\geq v$. When $\Delta S_{10}>0$, one has
	$$\re(\bar{z}_{10} \Delta S_{10}) = r_{10}\Delta P_{10} + x_{10}\Delta Q_{10}>0.$$
Similarly, when $\Delta S_{0,-1}>0$, one has $\re(\bar{z}_{10} \Delta S_{0,-1}) >0$. Then it follows from \eqref{construct v1} that
	\begin{eqnarray*}
	\Delta v_1 &=& 2\re(\bar{z}_{10} \Delta S_{10}) - |z_{10}|^2 \Delta \ell_{10} \\
	&>& \re(\bar{z}_{10} \Delta S_{10}) - |z_{10}|^2 \Delta \ell_{10} \\
	&=& \re(\bar{z}_{10} (\Delta S_{10}-z_{10}\Delta \ell_{10})) \\
	&=& \re(\bar{z}_{10} \Delta S_{0,-1}) > 0.
	\end{eqnarray*}
Similarly, one has
	$$\Delta v_2 = \Delta v_1 + \re(\bar{z}_{21}\Delta S_{21}) + \re(\bar{z}_{21}\Delta S_{10}) > \Delta v_1 > 0.$$
Hence, $v'>v$. This completes the proof of Claim \ref{claim C1}. $\hfill\Box$

\noindent{\it Proof of Claim \ref{claim C2}: } When C2 holds, it follows from Lemma \ref{lemma: v} that $v'\leq\hat{v}(s')=\hat{v}(s)\leq\overline{v}$.
$\hfill\Box$

\begin{remark}\label{rem: generalization}
Theorem \ref{thm: condition} still holds if there is an additional power injection constraint $s\in\mathcal{S}$ in OPF, where $\mathcal{S}$ can be an arbitrary set. This is because we set $s'=s$ in the construction of $w'$ (the initialization step), and therefore $s\in\mathcal{S}$ implies $s'\in\mathcal{S}$. Hence, the introduction of additional constraint $s\in\mathcal{S}$ does not affect the proof that $w'$ is feasible for SOCP and has a smaller objective value than $w$. As a result, Theorem \ref{thm: condition} still holds.
\end{remark}

\section{A modified OPF problem}\label{sec: modification}
C2 in Theorem \ref{thm: condition} depends on SOCP solutions and cannot be checked a priori. This drawback motivates us to impose additional constraint
	\begin{equation}\label{add}
	s \in \mathcal{S}_{\mathrm{volt}}
	\end{equation}
on OPF such that C2 holds automatically. Constraint \eqref{add} is equivalent to $\hat{v}_i(s)\leq\overline{v}_i$ for $i\in\hN^+$---$n$ affine constraints on $s$. Since $v_i\leq\hat{v}_i(s)$ (Lemma \ref{lemma: v}), the voltage upper bound constraints $v_i\leq\overline{v}_i$ in \eqref{OPF constraint v} do not bind after imposing \eqref{add}. To summarize, the modified OPF problem is
    	\begin{align}
    	\textbf{OPF-m:}~\min~~ & \sum_{i\in\hN}f_i(\re(s_i)) \nonumber\\
	\mathrm{over}~~ & s,S,v,\ell,s_0 \nonumber\\
	\mathrm{s.t.}~~ & \eqref{OPF S}-\eqref{OPF constraint s}; \nonumber\\
	& \underline{v}_i \leq v_i, ~\hat{v}_i(s)\leq\overline{v}_i, \quad i\in\hN^+.\label{modify}
    	\end{align}
A modification is necessary to ensure that SOCP is exact, since it is in general not exact otherwise. Remarkably, the feasible sets of OPF-m and OPF are similar since $\hat{v}_i(s)$ is close to $v_i$ in practice \cite{Baran89_capacitor_sizing,Baran89_network_reconfiguration,Turitsyn10}.

One can still relax \eqref{OPF ell} to \eqref{relax} to obtain a relaxation of OPF-m:
    \begin{align*}
    \textbf{SOCP-m:}~\min~~ & \sum_{i\in\hN}f_i(\re(s_i)) \nonumber\\
	\mathrm{over}~~ & s,S,v,\ell,s_0\nonumber\\
	\mathrm{s.t.}~~ & \eqref{OPF S}-\eqref{OPF v}, ~\eqref{relax}, ~\eqref{OPF constraint s}, ~\eqref{modify}.
    \end{align*}
Note again that SOCP-m is not necessarily convex, since we allow $f_i$ and $\mathcal{S}_i$ to be nonconvex.

Since OPF-m is obtained by imposing additional constraint \eqref{add} on OPF, it follows from Remark \ref{rem: generalization} that:
	\begin{theorem}\label{thm: exact}
	Assume that $f_0$ is strictly increasing, and that there exists $\overline{p}_i$ and $\overline{q}_i$ such that $\mathcal{S}_i\subseteq\{s\in\mathbb{C}~|~\re(s)\leq\overline{p}_i,~\im(s)\leq\overline{q}_i\}$ for $i\in\hN^+$. Then SOCP-m is exact if C1 holds.
	\end{theorem}
Theorem \ref{thm: exact} implies that after restricting the power injection $s$ to the region $\mathcal{S}_{\mathrm{volt}}$ where voltage upper bounds do not bind, the corresponding SOCP-m relaxation is exact under C1---a mild condition that can be checked a priori.

\begin{theorem}\label{thm: unique}
	If $f_i$ is convex for $i\in \hN$, $\mathcal{S}_i$ is convex for $i\in\hN^+$, and SOCP-m is exact, then SOCP-m has at most one solution.
	\end{theorem}
	Theorem \ref{thm: unique} implies that SOCP-m has at most one solution if it is convex and exact. It is proved in Appendix \ref{app: thm unique}.

\section{Connection with prior results}\label{sec: prior works}

We compare the sufficient condition in Theorem \ref{thm: condition} for the exactness of the SOCP relaxation
for radial networks with those  in the literature. 
As mentioned earlier there are mainly three categories of existing sufficient conditions:
\begin{enumerate}
\item The power injection constraints satisfy certain patterns  
\cite{Masoud11, Farivar-2013-BFM-TPS, Bose2011, Bose-2012-QCQPt, Zhang2013, Sojoudi2012PES, Sojoudi2013},
e.g., there are no lower bounds on the power injections (load over-satisfaction).

\item The phase angle difference across each line is bounded in terms of its $r/x$ ratio
\cite{Zhang2013, LavaeiTseZhang2012, lam2012distributed}.  

\item The voltage upper bounds are relaxed plus some other conditions \cite{Gan12,Gan13}.
\end{enumerate}

It is interesting to contrast the result in \cite{Farivar-2013-BFM-TPS} and  
Theorem \ref{thm: condition}.   The sufficient condition in \cite{Farivar-2013-BFM-TPS} relaxes
the lower bound on power injections but allows arbitrary constraints on the voltage magnitudes
whereas the condition in Theorem \ref{thm: condition} relaxes the upper bound on voltage magnitudes
but allows arbitrary constraints on power injections as long as they are upper bounded.  
As shown in Section \ref{sec: modification} voltage upper bounds can be imposed provided
we constrain the power injections.
The condition in \cite{Farivar-2013-BFM-TPS} requires the objective function be strictly
increasing in each $\ell_{ij}$ and nondecreasing in each $s_i$
whereas that in Theorem \ref{thm: condition} requires it be 
strictly increasing in $s_0$.

We now show that Theorem \ref{thm: condition} unifies and generalizes the results 
\cite{Gan12,Gan13} due to the following theorem proved in Appendix \ref{app: weaker}.
It says that C1 holds if at least one of the following holds: there is no distributed generation 
or shunt capacitors; lines use the same type of cable; there is no distributed generation 
and lines get thinner as they branch out; there are no shunt capacitors and lines get 
thicker as they branch out.
	\begin{theorem}\label{lemma: weaker}
	Assume that there exists $\overline{p}_i$ and $\overline{q}_i$ such that $\mathcal{S}_i\subseteq\{s\in\mathbb{C}~|~\re(s)\leq\overline{p}_i,~\im(s)\leq\overline{q}_i\}$ for $i\in\hN^+$. Then C1 holds if any one of the following statements is true:
	\begin{itemize}
	\item[(i)] $\hat{P}_{ij}(\overline{p})\leq0$, $\hat{Q}_{ij}(\overline{q})\leq0$ for any $(i,j)\in\hE$ such that $i\notin \hL$.
	
	\item[(ii)] $r_{ij}/x_{ij}=r_{jk}/x_{jk}$ for any $(i,j),(j,k)\in\hE$; and $\underline{v}_i-2r_{ij}\hat{P}_{ij}^+(\overline{p})-2x_{ij}\hat{Q}_{ij}^+(\overline{q})>0$ for any $(i,j)\in\hE$ such that $i\notin \hL$.
	
	\item[(iii)] $r_{ij}/x_{ij}\geq r_{jk}/x_{jk}$ for any $(i,j),(j,k)\in\hE$; and $\hat{P}_{ij}(\overline{p})\leq0$, $\underline{v}_i-2x_{ij}\hat{Q}_{ij}^+(\overline{q})>0$ for any $(i,j)\in\hE$ such that $i\notin \hL$.
	
	\item[(iv)] $r_{ij}/x_{ij}\leq r_{jk}/x_{jk}$ for any $(i,j),(j,k)\in\hE$; and $\hat{Q}_{ij}(\overline{q})\leq0$, $\underline{v}_i-2r_{ij}\hat{P}_{ij}^+(\overline{p})>0$ for any $(i,j)\in\hE$ such that $i\notin \hL$.
	
	\item[(v)] $\begin{pmatrix}
	\displaystyle \prod_{(k,l)\in \hP_j}\left( 1-\frac{2r_{kl}\hat{P}_{kl}^+(\overline{p})}{\underline{v}_k} \right) &
	\displaystyle -\sum_{(k,l)\in \hP_j}\frac{2r_{kl}\hat{Q}_{kl}^+(\overline{q})}{\underline{v}_k} \\
	\displaystyle -\sum_{(k,l)\in \hP_j}\frac{2x_{kl}\hat{P}_{kl}^+(\overline{p})}{\underline{v}_k} &
	\displaystyle \prod_{(k,l)\in \hP_j}\left( 1-\frac{2x_{kl}\hat{Q}_{kl}^+(\overline{q})}{\underline{v}_k} \right)
	\end{pmatrix}
	\begin{pmatrix}
	r_{ij} \\ x_{ij}
	\end{pmatrix} >0$ for all $(i,j)\in\hE$.
	\end{itemize}
	\end{theorem}
The results in \cite{Gan12,Gan13} say that, if there are no voltage upper bounds, i.e., $\overline{v}=\infty$, then SOCP is exact if any one of (i)--(v) holds. Note that C2 holds automatically when $\overline{v}=\infty$, and that C1 holds if any one of (i)--(v) holds according to Theorem \ref{lemma: weaker}. 
The following corollary follows immediately from Theorem \ref{thm: exact} and Theorem \ref{lemma: weaker}.
	\begin{corollary}
	Assume that $f_0$ is strictly increasing, and that there exists $\overline{p}_i$ and $\overline{q}_i$ such that $\mathcal{S}_i\subseteq\{s\in\mathbb{C}~|~\re(s)\leq\overline{p}_i,~\im(s)\leq\overline{q}_i\}$ for $i\in\hN^+$. Then SOCP-m is exact if any one of (i)--(v) holds.
	\end{corollary}

\section{Case Studies}\label{sec: case study}
In this section, we use several test networks to demonstrate that
\begin{enumerate}
\item SOCP is much more efficient to compute than SDP.
\item C1 holds. We will define a notion of {\it C1 margin} that quantifies how well C1 is satisfied, and show that the margin is sufficiently large for the test networks.
\item The feasible sets of OPF and OPF-m are similar. We will define a notion of {\it modification gap} that quantifies how different the feasible sets of OPF and OPF-m are, and show that the gap is small for the test networks.
\end{enumerate}

\subsection{Test networks}
Our test networks include modified IEEE 13-, 34-, 37-, 123-bus networks \cite{IEEE} and two real-world networks \cite{Masoud11,farivar2011optimal} in the service territory of Southern California Edison (SCE), a utility company in California, USA \cite{SCE}.

The IEEE networks are unbalanced three-phase radial networks with some elements (regulators, circuit switches, transformers, and distributed loads) that are not modeled in \eqref{PF}. Therefore we modify the networks as follows.
\begin{enumerate}
\item Assume that each bus has three phases and split its load uniformly among the three phases.
\item Assume that the three phases are decoupled so that the network becomes three identical single phase networks.
\item Model closed circuit switches as shorted lines and ignore open circuit switches. Model regulators as multiplying the voltages by constant factors.\footnote{The constant factors are taken to be 1.08 in the simulations.} Model transformers as lines with proper impedance.  Model the distributed load on a line as two identical spot loads, one at the each end of the line.
\end{enumerate}
The SCE networks, a 47-bus one and a 56-bus one, are shown in Fig. \ref{fig:circuit} with parameters given in Tables \ref{table:data} and \ref{table: 56}.
	\begin{figure*}[!htbp]
    \centering
    \includegraphics[scale=0.4]{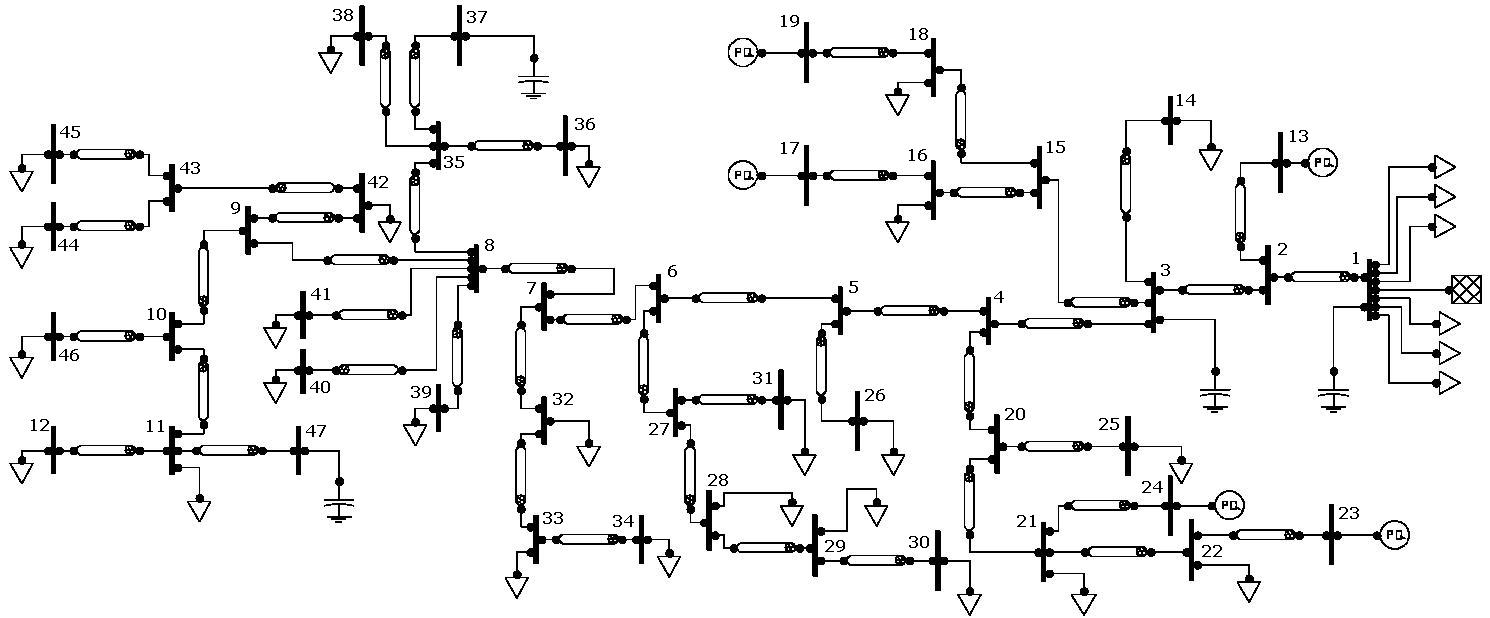}
    \includegraphics[scale=0.4]{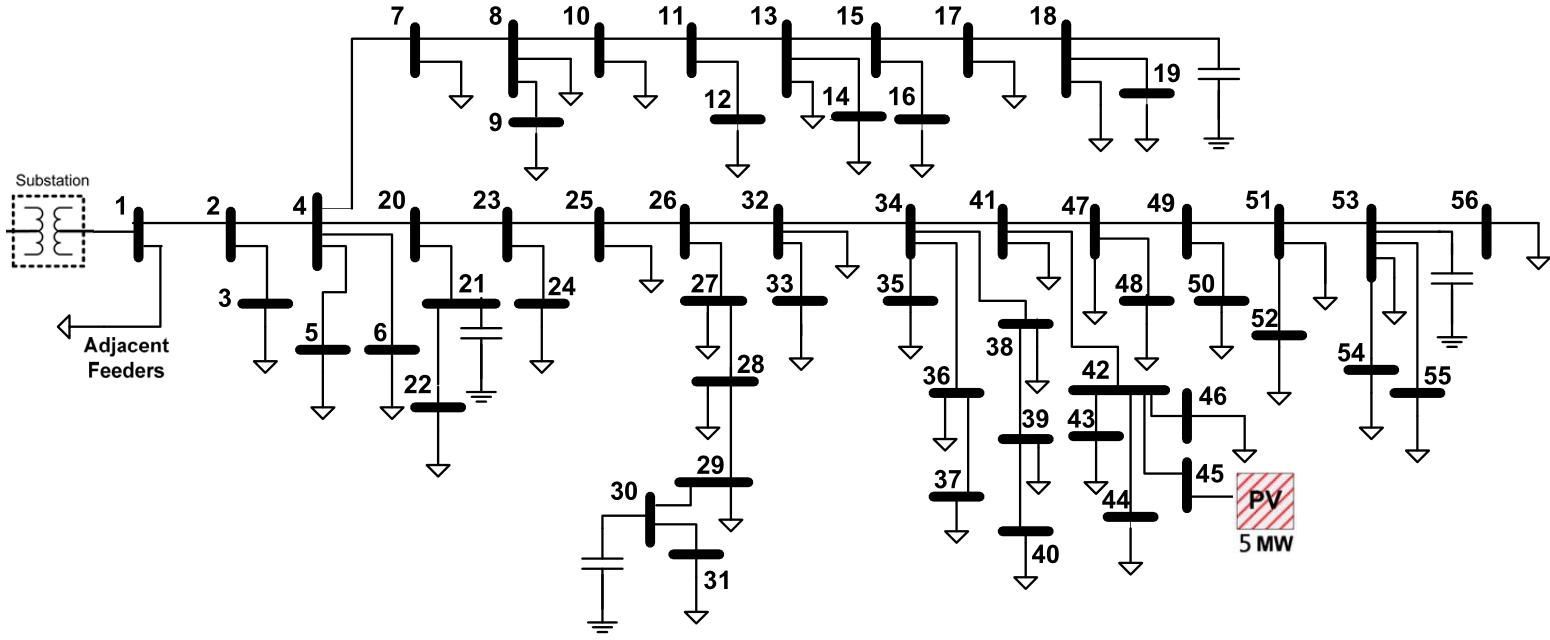}
    \caption{Topologies of the SCE 47-bus and 56-bus networks \cite{Masoud11,farivar2011optimal}.}
    \label{fig:circuit}
	\end{figure*}

\begin{table*}
\caption{Line impedances, peak spot load, and nameplate ratings of capacitors and PV generators of the 47-bus network.}
\centering
\scriptsize
\begin{tabular}{|p{0.42cm}|p{0.3cm}|m{0.43cm}|b{0.43cm}   |p{0.42cm}|p{0.3cm}|p{0.43cm}|p{0.43cm}    |p{0.42cm}|p{0.3cm}|p{0.43cm}|p{0.43cm} |p{0.3cm}|p{0.35cm}|c|c|c|c|}
\hline
\multicolumn{18}{|c|}{Network Data}\\
\hline
\multicolumn{4}{|c|}{Line Data}& \multicolumn{4}{|c|}{Line Data}& \multicolumn{4}{|c|}{Line Data}& \multicolumn{2}{|c|}{Load Data}& \multicolumn{2}{|c|}{Load Data}&\multicolumn{2}{|c|}{PV Generators}\\
\hline
From & To & R & X & From & To & R & X & From & To & R & X& Bus& Peak & Bus& Peak & Bus&Nameplate \\
Bus & Bus &$(\Omega)$& $(\Omega)$ & Bus & Bus & $(\Omega)$ & $(\Omega)$ & Bus & Bus & $(\Omega)$ & $(\Omega)$ & No &  MVA& No & MVA& No & Capacity\\
\hline
1	&	2	&	0.259	&	0.808	&	8	&	41	&	0.107	 &	0.031	&	 21	&	22	 &	 0.198	 &	 0.046	 &	1	&	30	&	34	&	0.2	&				 &		 \\	
2	&	13	&	0	&	0	&	8	&	35	&	0.076	 &	0.015	&	22	&	 23	&	0	 &	 0	 &	 11	&	 0.67	 &	36	&	0.27	&			 13	&	 1.5MW	\\	
2	&	3	&	0.031	&	0.092	&	8	&	9	&	0.031	 &	0.031	&	 27	&	31	 &	 0.046	 &	 0.015	 &	12	&	0.45	 &	38	&	0.45	 &			 17	&	 0.4MW	 \\	
3	&	4	&	0.046	&	0.092	&	9	&	10	&	0.015	 &	0.015	&	 27	&	28	 &	 0.107	 &	 0.031	 &	14	&	0.89	 &	39	&	1.34	 &			 19	&	 1.5 MW	 \\	
3	&	14	&	0.092	&	0.031	&	9	&	42	&	0.153	 &	0.046	&	 28	&	29	 &	 0.107	 &	 0.031	 &	16	&	0.07	 &	40	&	0.13	 &			 23	&	 1 MW	 \\	
3	&	15	&	0.214	&	0.046	&	10	&	11	&	0.107	 &	0.076	&	 29	&	30	 &	 0.061	 &	 0.015	 &	18	&	0.67	 &	41	&	0.67	 &			 24	&	 2 MW	 \\	
4	&	20	&	0.336	&	0.061	&	10	&	46	&	0.229	 &	0.122	&	 32	&	33	 &	 0.046	 &	 0.015	 &	21	&	0.45	 &	42	&	0.13	 &			 &		 \\\cline{17-18}		
4	&	5	&	0.107	&	0.183	&	11	&	47	&	0.031	 &	0.015	&	 33	&	34	 &	 0.031	 &	 0.010	 &	22	&	2.23	 &	44	&	0.45	&			 \multicolumn{2}{c|}{Shunt Capacitors} 			 \\  \cline{17-18}	
5	&	26	&	0.061	&	0.015	&	11	&	12	&	0.076	 &	0.046	&	 35	&	36	 &	 0.076	 &	 0.015	 &	25	&	0.45	 &	45	&	0.2	&			 \multicolumn{1}{c|}{Bus} &	 \multicolumn{1}{c|}{Nameplate}	\\		
5	&	6	&	0.015	&	0.031	&	15	&	18	&	0.046	 &	0.015	&	 35	&	37	 &	 0.076	 &	 0.046	 &	26	&	0.2	 &	46	&	0.45	&			 \multicolumn{1}{c|}{No.} &	 \multicolumn{1}{c|}{Capacity}	\\ \cline{15-18}		 
6	&	27	&	0.168	&	0.061	&	15	&	16	&	0.107	 &	0.015	&	 35	&	38	 &	 0.107	 &	 0.015	 &	28	&	0.13	 &	 \multicolumn{2}{c|}{ } 		&		&		 \\				
6	&	7	&	0.031	&	0.046	&	16	&	17	&	0	 &	0	&	42	&	 43	&	 0.061	 &	 0.015	&	 29	&	0.13	 &	 \multicolumn{2}{c|}{Base Voltage (kV) = 12.35}		 &	1	 &	 6000 kVAR	\\				
7	&	32	&	0.076	&	0.015	&	18	&	19	&	0	 &	0	&	43	&	 44	&	 0.061	 &	 0.015	&	 30	&	0.2	 &	\multicolumn{2}{c|}{Base kVA = 1000} 	 &	3	&	 1200 kVAR	 \\					
7	&	8	&	0.015	&	0.015	&	20	&	21	&	0.122	 &	0.092	&	 43	&	45	 &	 0.061	 &	 0.015	 &	31	&	0.07	 &	 \multicolumn{2}{c|}{Substation Voltage = 12.35}    &	 37 &		 1800 kVAR	 \\						
8	&	40	&	0.046	&	0.015	&	20	&	25	&	0.214	 &	0.046	&		 &		 &		 &		 &	 32	&	0.13	 &	\multicolumn{2}{c|}{}	 &	 47		 & 1800 kVAR  \\						
8	&	39	&	0.244	&	0.046	&	21	&	24	&	0	 &	0	&		&		 &		 &		 &	 33	&	 0.27	 &	\multicolumn{2}{c|}{}	& & \\
\hline
\end{tabular}
\label{table:data}
\end{table*}

\begin{centering}
\begin{table*}

\caption{Line impedances, peak spot load, and nameplate ratings of capacitors and PV generators
of the 56-bus network.}
\centering
\scriptsize
\begin{tabular}{|c|c|c|c|c|c|c|c|c|c|c|c|c|c|c|c|c|c|}
\hline
\multicolumn{18}{|c|}{Network Data}\\
\hline
\multicolumn{4}{|c|}{Line Data}& \multicolumn{4}{|c|}{Line Data}& \multicolumn{4}{|c|}{Line Data}& \multicolumn{2}{|c|}{Load Data}& \multicolumn{2}{|c|}{Load Data}&\multicolumn{2}{|c|}{Load Data}\\
\hline
From&To&R&X&From&To& R& X& From& To& R& X& Bus& Peak & Bus& Peak &  Bus &Peak  \\
Bus.&Bus.&$(\Omega)$& $(\Omega)$ & Bus. & Bus. & $(\Omega)$ & $(\Omega)$ & Bus.& Bus.& $(\Omega)$ & $(\Omega)$ & No.&  MVA& No.& MVA& No.& MVA\\
\hline

1	&	2	&	0.160	&	0.388	&	20	&	21	&	0.251	&	0.096	&	 39	&	40	 &	 2.349	 &	 0.964	&	3	&	0.057	&	29  &	0.044  & 52& 0.315	 \\
2	&	3	&	0.824	&	0.315	&	21	&	22	&	1.818	&	0.695	&	 34	&	41	 &	 0.115	 &	 0.278	&	5	&	0.121	&	31	&	0.053  & 54& 	 0.061	 \\
2	&	4	&	0.144	&	0.349	&	20	&	23	&	0.225	&	0.542	&	 41	&	42	 &	 0.159	 &	 0.384	&	6	&	0.049	&	32	&	0.223 & 55&	 0.055	 \\
4	&	5	&	1.026	&	0.421	&	23	&	24	&	0.127	&	0.028	&	 42	&	43	 &	 0.934	 &	 0.383	&	7	&	0.053	&	33	&	0.123 & 56&	 0.130	 \\\cline{17-18}
4	&	6	&	0.741	&	0.466   &	23	&	25	&	0.284	&	0.687	&	 42	&	44	 &	 0.506	 &	 0.163	&	8	&	0.047	&	34	&	0.067 & \multicolumn{2}{c|}{Shunt Cap}	 \\\cline{17-18}
4	&	7	&	0.528	&	0.468	&	25	&	26	&	0.171	&	0.414	&	 42	&	45	 &	 0.095	 &	 0.195	&	9	&	0.068	&	35	&	0.094&   \multicolumn{1}{c|}{Bus} &	 \multicolumn{1}{c|}{Mvar}			 \\\cline{17-18}
7	&	8	&	0.358	&	0.314	&	26	&	27	&	0.414	&	0.386	&	 42	&	46	 &	 1.915	 &	 0.769	&	10	&	0.048	&	36	&	0.097&  19& 	 0.6 	 \\
8	&	9	&	2.032	&	0.798	&	27	&	28	&	0.210	&	0.196	&	 41	&	47	 &	 0.157	 &	 0.379	&	11	&	0.067	&	37	&	0.281&  21&	 0.6 	 \\
8	&	10	&	0.502	&	0.441	&	28	&	29	&	0.395	&	0.369	&	 47	&	48	 &	 1.641	 &	 0.670	&	12	&	0.094	&	38	&	0.117&  30&	 0.6 		 \\
10	&	11	&	0.372	&	0.327	&	29	&	30	&	0.248	&	0.232	&	 47	&	49	 &	 0.081	 &	 0.196	&	14	&	0.057	&	39	&	0.131& 53&	 0.6 		 \\\cline{17-18}
11	&	12	&	1.431	&	0.999	&	30	&	31	&	0.279	&	0.260	&	 49	&	50	 &	 1.727	 &	 0.709	&	16	&	0.053	&	40	&	0.030& \multicolumn{2}{c|}{Photovoltaic}		 \\\cline{17-18}
11	&	13	&	0.429	&	0.377	&	26	&	32	&	0.205	&	0.495	&	 49	&	51	 &	 0.112	 &	 0.270	&	17	&	0.057	&	41	&	0.046& \multicolumn{1}{c|}{Bus} &	 \multicolumn{1}{c|}{Capacity}		 \\\cline{17-18}
13	&	14	&	0.671	&	0.257	&	32	&	33	&	0.263	&	0.073	&	 51	&	52	 &	 0.674	 &	 0.275	&	18	&	0.112	&	42	&	0.054&   & \\
13	&	15	&	0.457	&	0.401	&	32	&	34	&	0.071	&	0.171	&	 51	&	53	 &	 0.070	 &	 0.170	&	19	&	0.087	&	43	&	0.083&   45 &		 5MW		 \\\cline{17-18}
15	&	16	&	1.008	&	0.385	&	34	&	35	&	0.625	&	0.273	&	 53	&	54	 &	 2.041	 &	 0.780	&	22	&	0.063	&	44	&	0.057&  \multicolumn{2}{c|}{} \\
15	&	17	&	0.153	&	0.134	&	34	&	36	&	0.510	&	0.209	&	 53	&	55	 &	 0.813	 &	 0.334	&	24	&	0.135	&	46	&	0.134&  \multicolumn{2}{c|}{$V_\textrm{base}$ = 12kV}	 \\
17	&	18	&	0.971	&	0.722	&	36	&	37	&	2.018	&	0.829	&	 53	&	56	 &	 0.141	 &	 0.340	&	25	&	0.100	&	47	&	0.045& \multicolumn{2}{c|}{$S_\textrm{base}$ = 1MVA} 	 \\
18	&	19	&	1.885	&	0.721	&	34	&	38	&	1.062	&	0.406	&		 &		 &		 &		 &	 27	&	0.048	&	48	&	0.196&  \multicolumn{2}{c|}{$Z_\textrm{base}= 144 \Omega$ }		 \\
4	&	20	&	0.138	&	0.334	&	38	&	39	&	0.610	&	0.238	&		 &		 &		 &		 &	 28	&	0.038	&	50	&	0.045 &  \multicolumn{2}{c|}{}		 \\
		
\hline

\end{tabular}
\label{table: 56}
\end{table*}
\end{centering}

These networks have increasing penetration of distributed generation (DG). While the IEEE networks do not have any DG, the SCE 47-bus network has 56.6\% DG penetration ($6.4$MW nameplate distributed generation capacity against $11.3$MVA peak spot load) \cite{Masoud11}, and the SCE 56-bus network has 130.4\% DG penetration ($5$MW nameplate distributed generation capacity against $3.835$MVA peak spot load) \cite{farivar2011optimal} as listed in Table \ref{table: epsilon}.

	\begin{table}[!h]
	\caption{DG penetration, C1 margins, modification gaps, and computation times for different test networks.}
	\label{table: epsilon}
	\begin{center}
        \begin{tabular}{| c | c | c | c | c | c | c |}
        \hline
        & DG penetration & numerical precision & SOCP time & SDP time & C1 margin & estimated modification gap \\
        \hline
        IEEE 13-bus & 0\% & $10^{-10}$ & 0.5162s & 0.3842s & 27.6762 & 0.0362 \\
        IEEE 34-bus & 0\% & $10^{-10}$ & 0.5772s & 0.5157s & 20.8747 & 0.0232 \\
        IEEE 37-bus & 0\% & $10^{-9}$ & 0.5663s & 1.6790s & $+\infty$ & 0.0002 \\
        IEEE 123-bus & 0\% & $10^{-8}$ & 2.9731s & 32.6526s & 52.9636 & 0.0157 \\
        SCE 47-bus & 56.6\% & $10^{-8}$ & 0.7265s & 2.5932s & 2.5416 & 0.0082 \\
        SCE 56-bus & 130.4\% & $10^{-9}$ & 1.0599s & 6.0573s & 1.2972 & 0.0053 \\
        \hline
        \end{tabular}
        \end{center}
        \end{table}

\subsection{SOCP is more efficient to compute than SDP}
We compare the computation times of SOCP and SDP for the test networks, and summarize the results in Table \ref{table: epsilon}. All simulations in this paper use matlab 7.9.0.529 (64-bit) with toolbox cvx 1.21 on Mac OS X 10.7.5 with 2.66GHz Intel Core 2 Due CPU and 4GB 1067MHz DDR3 memory.

We use the following OPF setup throughout the simulations.
\begin{enumerate}
\item The objective is minimizing power loss in the network.
\item The power injection constraint is as follows. For each bus $i\in\hN^+$, there may be multiple devices including loads, capacitors, and PV panels. Assume that there is a total of $A_i$ such devices and label them by $1,2,\ldots,A_i$. Let $s_{i,a}$ denote the power injection of device $a$ for $a=1,2,\ldots,A_i$. If device $a$ is a load with given real and reactive power consumptions $p$ and $q$, then we impose
	\begin{equation}\label{load}
	s_{i,a} = -p-\ii q.
	\end{equation}
If device $a$ is a load with given peak apparent power $s_\mathrm{peak}$, then we impose
	\begin{equation}\label{load2}
	s_{i,a}=-s_{\mathrm{peak}}\exp(j\theta)
	\end{equation}
where $\theta=\cos^{-1}(0.9)$, i.e, power injection $s_{i,a}$ is considered to be a constant, obtained by assuming a power factor of 0.9 at peak apparent power. If device $a$ is a capacitor with nameplate $\overline{q}$, then we impose
	\begin{equation}\label{capacitor}
	\re(s_{i,a})=0 \text{ and } 0\leq\im(s_{i,a})\leq\overline{q}.
	\end{equation}
If device $a$ is a PV panel with nameplate $\overline{s}$, then we impose
	\begin{equation}\label{PV panel}
	\re(s_{i,a})\geq0 \text{ and } |s_{i,a}|\leq\overline{s}.
	\end{equation}
The power injection at bus $i$ is
	\[s_i=\sum_{a=1}^{A_i} s_{i,a}\]
where $s_{i,a}$ satisfies one of \eqref{load}--\eqref{PV panel}.
\item The voltage regulation constraint is considered to be $0.9^2\leq v_i\leq 1.1^2$ for $i\in\hN^+$. Note that we choose a small voltage lower bound 0.9 so that OPF is feasible for all test networks. We choose a big voltage upper bound 1.1 such that Condition C2 holds, and SDP/SOCP is exact if Condition C1 holds.
\end{enumerate}

    \begin{figure*}[!htbp]
    \centering
    \includegraphics[scale=0.4]{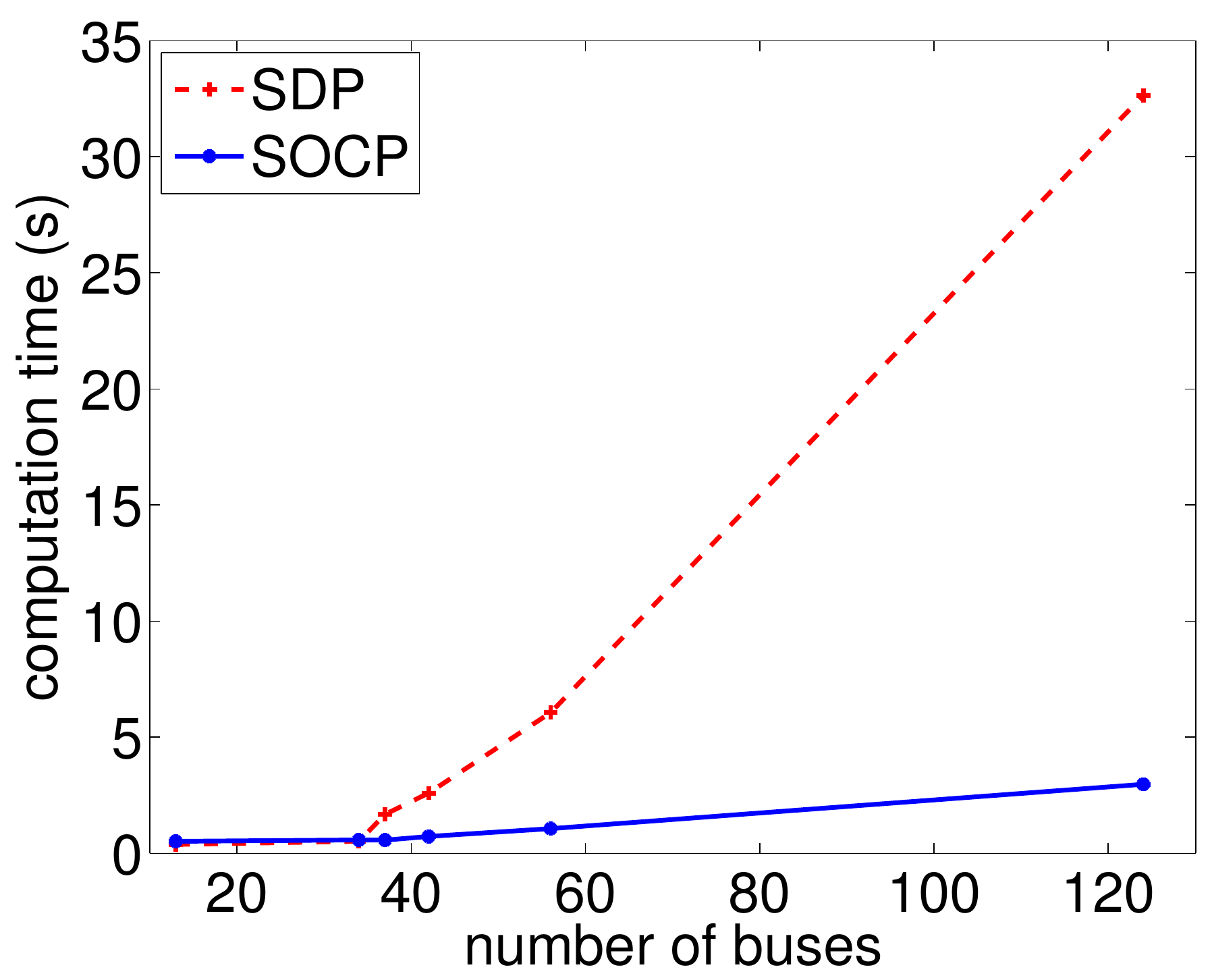}
    \includegraphics[scale=0.4]{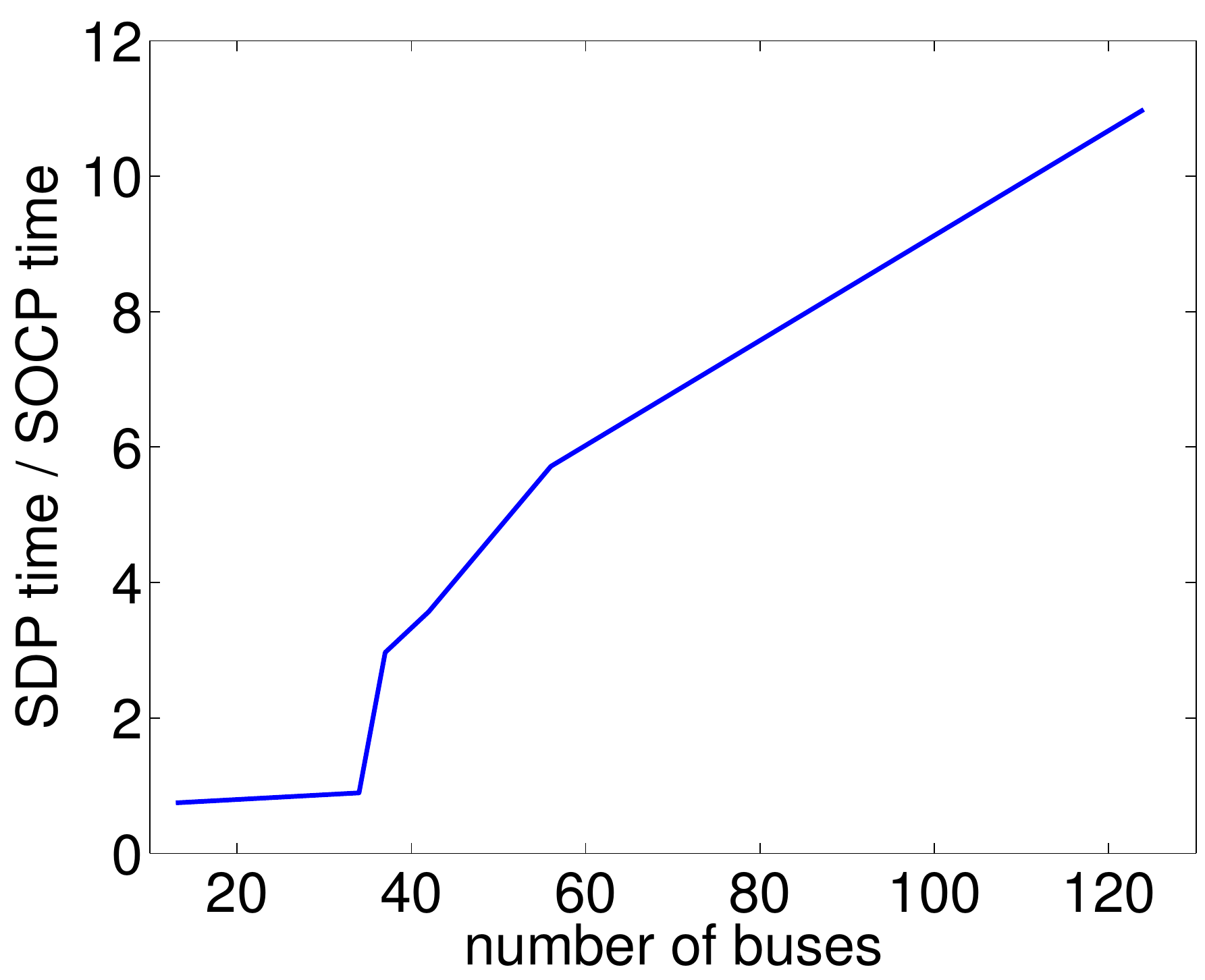}
    \caption{Comparison of the computation times for SOCP and SDP.}
    \label{fig: time}
    \end{figure*}
It can be seen from Fig. \ref{fig: time} that the computation time of SOCP scales up much more slowly than that of SDP as the number of buses increases, and that the improvement in efficiency (i.e., the ratio of SDP computation time to SOCP computation time) increases dramatically as the number of buses increases. Hence, even though the computation times of SOCP and SDP are similar for small networks, we expect SOCP to be much more efficient for medium to large networks.

SOCP and SDP can only be solved to certain numerical precision. The best numerical precision (without applying pre-conditioning techniques) that can be obtained by our simulation platform are listed in Table \ref{table: epsilon}. 

\subsection{C1 holds with a large margin}
We show that C1 holds with a large margin for all test networks. Recall that C1 is more difficult to satisfy as $(\overline{p},\overline{q})$ increases (Proposition \ref{prop: smaller injections}). One can scale up distributed generation (positive component of $\overline{p}$, $\overline{q}$) and shunt capacitors (positive component of $\overline{q}$) until C1 breaks down, and call the scaling factor when this happens the {\it C1 margin}. More specifically, for any scaling factor $\eta\geq0$, set
	\begin{eqnarray*}
	\overline{p}_i(\eta) &:=& \text{real load at }i + \eta * \text{PV nameplate at }i\\
	\overline{q}_i(\eta) &:=& \text{reactive load at }i + \eta * (\text{PV nameplate at }i + \text{shunt capacitor nameplate at }i)
	\end{eqnarray*}
for $i\in\hN^+$. When $\eta=0$, one has $(\overline{p}(\eta),\overline{q}(\eta))\leq0$ and therefore C1 holds according to Proposition \ref{pro: no reverse}. According to Proposition \ref{prop: smaller injections}, there exists a unique $\eta^*\in\mathbb{R}^+\cup\{+\infty\}$, such that
    \begin{subequations}\label{C1 margin}
	\begin{align}
	0\leq\eta < \eta^* \Rightarrow& \text{ C1 holds for }(r,x,\overline{p}(\eta),\overline{q}(\eta),\underline{v});\\
	\eta > \eta^* \Rightarrow& \text{ C1 does not hold for }(r,x,\overline{p}(\eta),\overline{q}(\eta),\underline{v}).
	\end{align}
    \end{subequations}

    \begin{definition}
    C1 margin is defined as the unique $\eta^*\geq0$ that satisfies \eqref{C1 margin}.
    \end{definition}
Physically, $\eta^*$ is the number of multiples one can scale up distributed generation and shunt capacitors before C1 breaks down. Noting that $\overline{p}=\overline{p}(1)$ and $\overline{q}=\overline{q}(1)$, C1 holds for $(r,x,\overline{p},\overline{q},\underline{v})$ if and only if $\eta^*>1$ (ignore the corner case where $\eta^*=1$). The larger $\eta^*$ is, the ``more safely'' C1 holds. The C1 margins for different test networks are summarized in Table \ref{table: epsilon}. The minimum C1 margin is 1.30, meaning that one can scale up distributed generation and shunt capacitors by 1.39 before C1 breaks down. C1 margin of the IEEE 37-bus network is $+\infty$, and this is because there is no distributed generation or shunt capacitors in the IEEE 37-bus network.

C1 margin is above 10 for all IEEE networks, but much smaller for SCE networks. This is because SCE networks have high penetration of distributed generation---big positive $\overline{p},\overline{q}$---that makes C1 more difficult to hold. On the other hand, the SCE 56-bus network already has a DG penetration of over 130\%, and one can still scale up DG by a factor of 1.30 before C1 breaks down. This finishes the demonstration that C1 is mild.

\subsection{The feasible sets of OPF and OPF-m are similar}
\label{sec: epsilon}
We show that OPF-m eliminates some feasible points of OPF that are close to the voltage upper bounds for each of the test networks. Let $\hF_{\text{OPF}}$ denote the feasible set of OPF and let $\|\cdot\|_{\infty}$ denote the $\ell_{\infty}$ norm.\footnote{The $\ell_\infty$ norm of a vector $x=(x_1,\ldots,x_n)\in\mathbb{R}^n$ is $\|x\|_\infty = \max\{|x_1|,\ldots,|x_n|\}$.} Define
	\begin{align} \label{eps}
	\eps ~\eqdef~ \max &~ \|\hat{v}(s)-v\|_\infty\\
	\text{s.t.} &~ (s,S,v,\ell,s_0)\in\hF_{\text{OPF}} \nonumber
	\end{align}
as the maximum deviation of $v$ from $\hat{v}$ over all OPF feasible points.

The value $\eps$ serves as a measure for the difference between the feasible sets of OPF and OPF-m. Consider the OPF problem with a stricter voltage upper bound constraint:
	\begin{eqnarray*}
	\textbf{OPF-$\eps$: }\min && \sum_{i\in\hN}f_i(\re(s_i)) \\
	\mathrm{over}  & & s,S,v,\ell,s_0\nonumber\\
	\mathrm{s.t.} && \eqref{OPF S}-\eqref{OPF constraint s};\\
	&& \underline{v}_i \leq v_i \leq \overline{v}_i-\varepsilon, \quad i\in \hN^+.
	\end{eqnarray*}
The feasible set $\hF_{\text{OPF-}\eps}$ is contained in $\hF_{\text{OPF}}$, and therefore
	\[\hat{v}_i(s) \leq v_i + \eps \leq \overline{v}_i-\eps+\eps = \overline{v}_i \text{ for }i\in\hN^+\]
for every $(s,S,\ell,v,s_0)\in\hF_{\text{OPF-}\eps}$ according to the definition of $\eps$. It follows that $\hF_{\text{OPF-}\eps} \subseteq \hF_{\text{OPF-m}}$ and therefore
	\[\hF_{\text{OPF-}\varepsilon}\subseteq\hF_\text{OPF-m}\subseteq\hF_\text{OPF}\]
as in Fig. \ref{fig: MOPF}. If $\eps$ is small, then $\hF_\text{OPF-m}$ is similar to $\hF_\text{OPF}$.
	\begin{figure}[!htbp]
     	\centering
     	\includegraphics[scale=0.4]{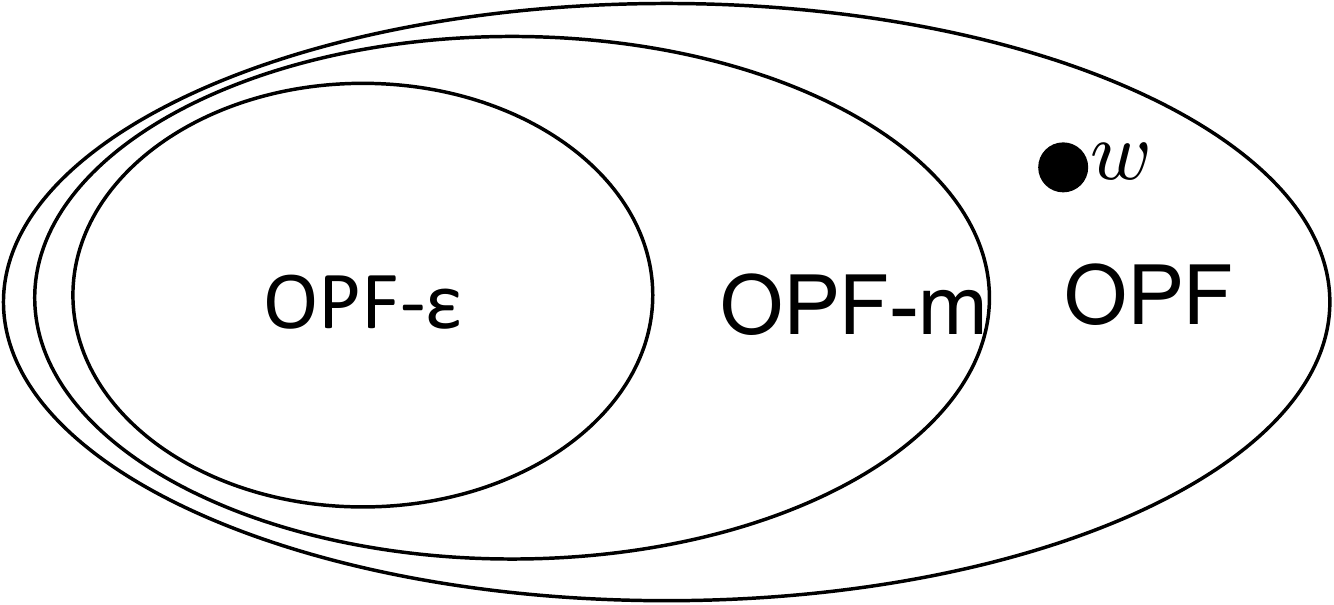}
      	\caption{Feasible sets of OPF-$\varepsilon$, OPF-m, and OPF. The point $w$ is feasible for OPF but not for OPF-m. }
      	\label{fig: MOPF}
	\end{figure}
Furthermore, any point $w$ that is feasible for OPF but infeasible for OPF-m is close to the voltage upper bound since $v_i>\overline{v}_i-\varepsilon$ for some $i\in \hN^+$. Such points are perhaps undesirable for robust operation.

\begin{definition}
The value $\eps$ defined in \eqref{eps} is called the {\it modification gap}.
\end{definition}

We demonstrate that the modification gap $\eps$ is small for all test networks through Monte-Carlo simulations. Note that $\eps$ is difficult to compute since the objective function in \eqref{eps} is not concave and the constraints in \eqref{eps} are not convex. We choose $1000$ samples of $s$, calculate the corresponding $(S,v,\ell,s_0)$ by solving power flow \eqref{BFM S}--\eqref{BFM ell} (using the {\it forward backward sweep} algorithm \cite{kersting2012distribution}) for each $s$, and compute $\eps(s):=\|\hat{v}(s)-v\|_\infty$ if $(s,S,v,\ell,s_0)\in\hF_{\text{OPF}}$. We use the maximum $\eps(s)$ over the samples as an estimate for $\eps$. The estimated modification gap $\eps^{\text{set}}$ we obtained for different test networks are listed in Table \ref{table: epsilon}. For example, $\eps^{\text{set}}=0.0362$ for the IEEE 13-bus network, in which case the voltage constraints are $0.81\leq v_i\leq 1.21$ for OPF and $0.81\leq v_i\leq 1.1738$ for OPF-$\eps$ (assuming $\eps=\eps^{\text{set}}$).
\section{Conclusion}
We have proved that SOCP is exact if Conditions C1 and C2 hold. C1 can be checked a priori, and follows from the physical intuition that all upstream power flows should increase if the power loss on a line is reduced. C2 requires that optimal power injections lie in a region ($\mathcal{S}_{\mathrm{volt}}$) where voltage upper bounds do not bind. C2 depends on SOCP solutions and cannot be checked a priori, but holds automatically after imposing the additional constraint that power injections lie in $\mathcal{S}_{\mathrm{volt}}$. This result unifies and generalizes our prior works \cite{Gan12,Gan13}.

We have proposed a modified OPF problem by imposing the additional constraint that power injections lie in $\mathcal{S}_{\mathrm{volt}}$ such that C2 holds automatically. The modified OPF problem has an exact SOCP relaxation if C1 holds. We have also proved that SOCP has at most one solution if it it convex and exact.

Empirical studies have verified that SOCP is computationally efficient, that C1 holds with large margin, and that the feasible sets of OPF and OPF-m are close for the IEEE 13-, 34-, 37-, 123-bus networks and two real-world networks.


\bibliographystyle{IEEEtran}
\bibliography{TAC_2013}


\appendices
\section{Proof of Lemma \ref{lemma: v}}\label{app: lemma v}
Let $(s,S,v,\ell,s_0)$ satisfy \eqref{BFM S}--\eqref{BFM v} and $\ell\geq0$ componentwise. It follows from \eqref{BFM S} that
	$$S_{ij} = s_i + \sum_{h:\,h\rightarrow i} (S_{hi}-z_{hi}\ell_{hi}) \leq s_i + \sum_{h:\,h\rightarrow i} S_{hi}$$
for $(i,j)\in\hE$. On the other hand, $\hat{S}_{ij}(s)$ is the solution of
	$$\hat{S}_{ij} = s_i + \!\!\!\sum_{h:\,h\rightarrow i} \hat{S}_{hi}$$
for $(i,j)\in\hE$. By induction from the leaf lines, one can show that $S_{ij} \leq \hat{S}_{ij}(s)$ for $(i,j)\in\hE$.

It follows from \eqref{BFM v} that 
	$$v_i-v_j = 2\re(\bar{z}_{ij}S_{ij})-|z_{ij}|^2\ell_{ij} \leq 2\re(\bar{z}_{ij}S_{ij}) \leq 2\re(\bar{z}_{ij}\hat{S}_{ij}(s))$$
for $(i,j)\in\hE$. Sum up the inequalities over $\hP_i$ to obtain
	$$v_i-v_0 \leq 2\sum_{(j,k)\in\hP_i}\re(\bar{z}_{jk}\hat{S}_{jk}(s)),$$
i.e., $v_i\leq \hat{v}_i(s)$, for $i\in \hN$.

\section{Proof of Theorem \ref{thm: condition}}\label{app: condition}
The proof idea of Theorem \ref{thm: condition} has been illustrated via a 3-bus network in Section \ref{sec: idea}. Now we present the proof of Theorem \ref{thm: condition} for general tree networks. Assume that $f_0$ is strictly increasing, and that C1 and C2 hold. If SOCP is not exact, then there exists an SOCP solution $w=(s,S,v,\ell,s_0)$ that violates \eqref{OPF ell}. We will construct another feasible point $w'=(s',S',v',\ell',s_0')$ of SOCP that has a smaller objective value than $w$. This contradicts the optimality of $w$, and therefore SOCP is exact.

\subsection*{Construction of $w'$}
The construction of $w'$ is as follows. Since $w$ violates \eqref{OPF ell}, there exists a leaf bus $l\in\hL$ with $m\in\{1,\ldots,n_l\}$ such that $w$ satisfies \eqref{OPF ell} on $(l_1,l_0),\ldots,(l_{m-1},l_{m-2})$ and violates \eqref{OPF ell} on $(l_m,l_{m-1})$. Without loss of generality, assume $l_k=k$ for $k=0,\ldots,m$ as in Fig. \ref{fig: tree}. Then
	\begin{equation}\label{m}
	\ell_{m,m-1} > \frac{|S_{m,m-1}|^2}{v_m}, \quad
	\ell_{k,k-1}    = \frac{|S_{k,k-1}|^2}{v_k} \text{ for }k=1,\ldots,m-1.
	\end{equation}
	
	\begin{figure}[!htbp]
     	\centering
     	\includegraphics[scale=0.4]{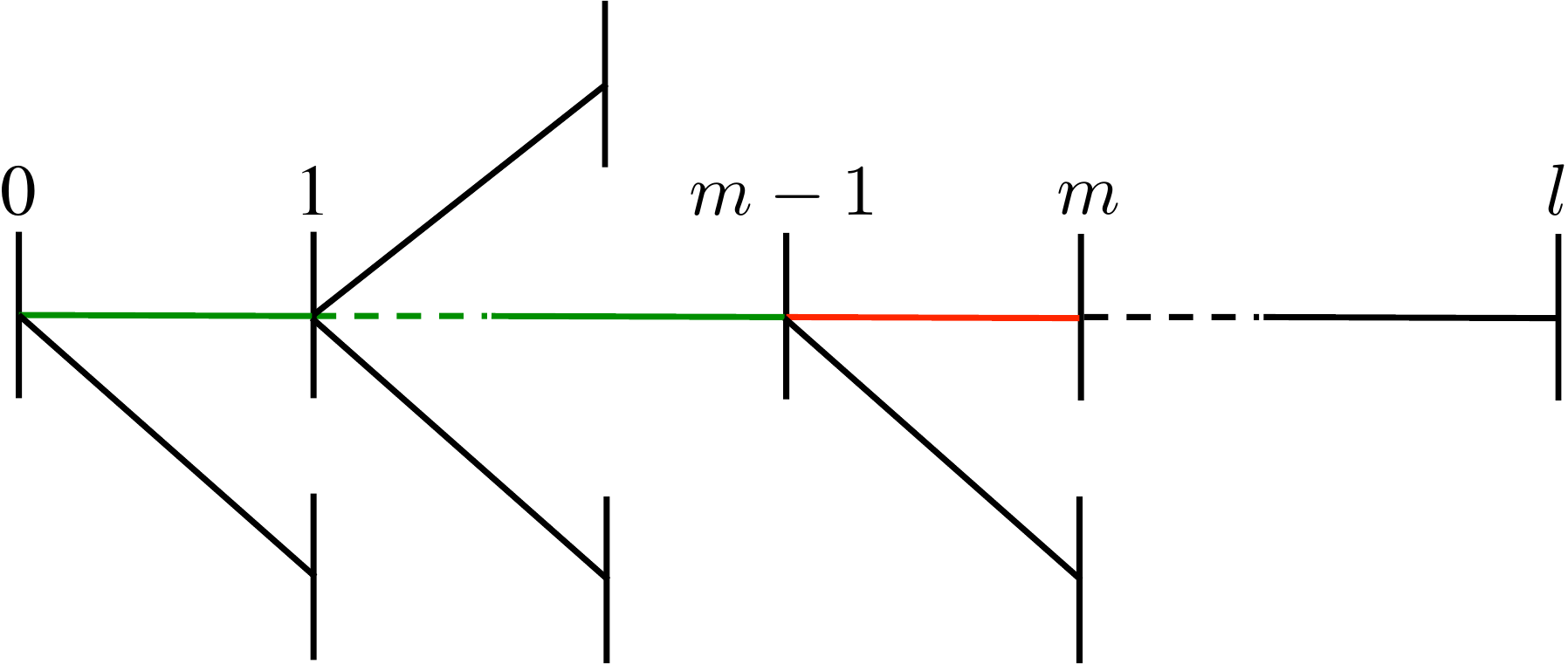}
      	\caption{Bus $l$ is a leaf bus, with $l_k=k$ for $k=0,\ldots,m$. Equality \eqref{OPF ell} is satisfied on $[0,m-1]$, but violated on $[m-1,m]$.}
      	\label{fig: tree}
	\end{figure}
	
One can then construct $w'=(s',S',v',\ell',s_0')$ as in Algorithm \ref{algorithm}.
	\begin{algorithm}
	\caption{Construct a feasible point}
	\label{algorithm}
	\begin{algorithmic}[1]
	
    	\REQUIRE an SOCP solution $w=(s,S,v,\ell,s_0)$ that violates \eqref{OPF ell}, a leaf bus $l\in\hL$ with $1\leq m\leq n_l$ such that \eqref{m} holds (assume $l_k=k$ for $k=0,\ldots,m$ without loss of generality).
	
    	\ENSURE $w'=(s',S',v',\ell',s_0')$.
	
    	\STATE Initialization. (Construct $s'$, $\ell'$ outside $\hP_m$, and $S'$ outside $\hP_{m-1}$.)\\
	keep $s$: $s'\leftarrow s$;\\
	keep $\ell$ outside path $\hP_{m}$: $\ell_{ij}'\leftarrow \ell_{ij}$ for $(i,j)\notin\hP_{m}$;\\
	keep $S$ outside path $\hP_{m-1}$: $S_{ij}'\leftarrow S_{ij}$ for $(i,j)\notin\hP_{m-1}$;
	
	\STATE Forward sweep. (Construct $\ell'$ on $\hP_{m}$, $S'$ on $\hP_{m-1}$, and $s_0'$.)\\
	{\bf for} $k=m,m-1,\ldots,1$ {\bf do}\\
    		$\qquad\ell_{k,k-1}' \leftarrow |S_{k,k-1}'|^2 / v_k$;\\
		$\qquad S_{k-1,k-2}' \leftarrow s_{k-1}\mathbbm{1}_{k\neq1} + \sum_{j:\,j\rightarrow k-1} (S_{j,k-1}'-z_{j,k-1}\ell_{j,k-1}')$;\\
	{\bf end for}\\
	$s_0'\leftarrow - S_{0,-1}'$;
	
	\STATE Backward sweep. (Construct $v'$.)\\
	$v_0'\leftarrow v_0$;\\
	$\hN_{\mathrm{visit}}=\{0\}$;\\
	{\bf while} $\hN_{\mathrm{visit}} \neq \hN$ {\bf do}\\
		$\qquad$find $i\notin\hN_{\mathrm{visit}}$ and $j\in\hN_{\mathrm{visit}}$ such that $i\rightarrow j$;\\
		$\qquad v_i'\leftarrow v_j'+2\re(\bar{z}_{ij}S_{ij}')-|z_{ij}|^2\ell_{ij}'$;\\
		$\hN_\mathrm{visit} \leftarrow \hN_\mathrm{visit}\cup\{i\}$;\\
	{\bf end while}
	\end{algorithmic}
	\end{algorithm}
The construction consists of three steps:
\begin{itemize}
\item[S1] In the initialization step, $s'$, $\ell'$ outside path $\hP_m$, and $S'$ outside path $\hP_{m-1}$ are initialized as the corresponding values in $w$. Since $s'=s$, the point $w'$ satisfies \eqref{OPF constraint s}. Furthermore, since $\ell_{ij}'=\ell_{ij}$ for $(i,j)\notin\hP_m$ and $S_{ij}'=S_{ij}$ for $(i,j)\notin\hP_{m-1}$, the point $w'$ also satisfies \eqref{OPF S} for $(i,j)\notin\hP_{m-1}$.
\item[S2] In the forward sweep step, $\ell_{k,k-1}'$ and $S_{k-1,k-2}'$ are recursively constructed for $k=m,\ldots,1$ by alternatively applying \eqref{OPF ell} (with $v'$ replaced by $v$) and \eqref{OPF S}/\eqref{OPF s}. Hence, $w'$ satisfies \eqref{OPF S} for $(i,j)\in\hP_{m-1}$ and \eqref{OPF s}.
\item[S3] In the backward sweep step, $v_i'$ is recursively constructed from bus 0 to leaf buses by applying \eqref{OPF v} consecutively. Hence, the point $w'$ satisfies \eqref{OPF v}.
\end{itemize}
The point $w'$ satisfies another important property given below.
\begin{lemma}\label{lemma: two points}
The point $w'$ satisfies $\ell_{ij}'\geq|S_{ij}'|^2/v_i$ for $(i,j)\in\hE$.
\end{lemma}
\begin{proof}
When $(i,j)\notin\hP_m$, it follows from Step S1 that $\ell_{ij}'=\ell_{ij}\geq|S_{ij}|^2/v_i=|S_{ij}'|^2/v_i$. When $(i,j)\in\hP_m$, it follows from Step S2 that $\ell_{ij}'=|S_{ij}'|^2/v_i$. This completes the proof of Lemma \ref{lemma: two points}.
\end{proof}
Lemma \ref{lemma: two points} implies that if $v'\geq v$, then $w'$ satisfies \eqref{relax}.

\subsection*{Feasibility and Superiority of $w'$}
We will show that $w'$ is feasible for SOCP and has a smaller objective value than $w$. This result follows from Claims \ref{claim: C1 tree} and \ref{claim: C2 tree}.
\begin{claim}\label{claim: C1 tree}
$\text{C1} ~\Rightarrow~ S_{k,k-1}'>S_{k,k-1} \text{ for } k=0,\ldots,m-1 ~\Rightarrow~ v'\geq v$.
\end{claim}
Claim \ref{claim: C1 tree} is proved later in this appendix. Here we illustrate with Fig. \ref{fig: illustrate} that $S_{k,k-1}'>S_{k,k-1} \text{ for } k=0,\ldots,m-1$ seems natural to hold.	
	\begin{figure}[!htbp]
     	\centering
     	\includegraphics[scale=0.4]{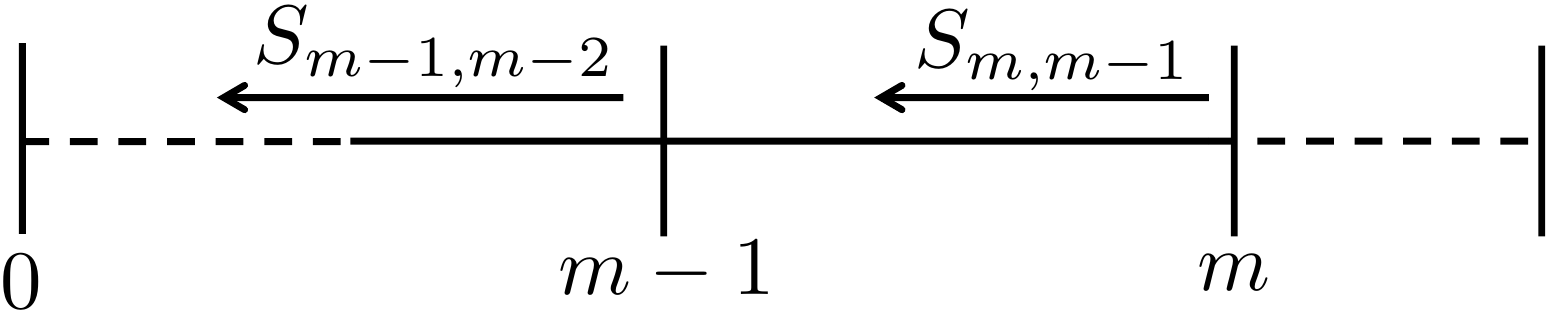}
      	\caption{Illustration of $S_{k,k-1}'>S_{k,k-1} \text{ for } k=0,\ldots,m-1$.}
      	\label{fig: illustrate}
	\end{figure}
Note that $S_{m,m-1}'=S_{m,m-1}$ and that $\ell_{m,m-1}'=|S_{m,m-1}'|^2/v_m = |S_{m,m-1}|^2/v_m < \ell_{m,m-1}$. Define $\Delta w=(\Delta s,\Delta S,\Delta v,\Delta \ell, \Delta s_0)=w'-w$, then $\Delta \ell_{m,m-1}<0$ and therefore
	\begin{equation}
	\Delta S_{m-1,m-2} = \Delta S_{m,m-1}-z_{m,m-1}\Delta\ell_{m,m-1} = -z_{m,m-1}\Delta\ell_{m,m-1} > 0. \label{m difference}
	\end{equation}
Intuitively, after increasing $S_{m-1,m-2}$, upstream reverse power flow $S_{k,k-1}$ is likely to increase for $k=0,\ldots,m-2$. C1 is a condition that ensures $S_{k,k-1}$ to increase for $k=0,\ldots,m-1$.

\begin{claim}\label{claim: C2 tree}
$\text{C2} ~\Rightarrow~ v'\leq \overline{v}$.
\end{claim}
\begin{proof}
When C2 holds, it follows from Lemma \ref{lemma: v} that $v'\leq \hat{v}(s')=\hat{v}(s)\leq\overline{v}$.
\end{proof}

It follows from Claims \ref{claim: C1 tree} and \ref{claim: C2 tree} that $\underline{v}\leq v\leq v'\leq\overline{v}$, and therefore $w'$ satisfies \eqref{OPF constraint v}. Besides, it follows from Lemma \ref{lemma: two points} that $\ell_{ij}'\geq|S_{ij}'|^2/v_i\geq|S_{ij}'|^2/v_i'$ for $(i,j)\in\hE$, i.e., $w'$ satisfies \eqref{relax}. Hence, $w'$ is feasible for SOCP. Furthermore, $w'$ has a smaller objective value than $w$ because
	\begin{eqnarray*}
	\sum_{i\in\hN} f_i(\re(s_i')) - \sum_{i\in\hN} f_i(\re(s_i)) = f_0(-\re(S_{0,-1}')) - f_0(-\re(S_{0,-1})) < 0.
	\end{eqnarray*}
This contradicts with the optimality of $w$, and therefore SOCP is exact. To complete the proof, we are left to prove Claim \ref{claim: C1 tree}.

\subsection*{Proof of Claim \ref{claim: C1 tree}}
First show that C1 implies $\Delta S_{k,k-1}>0$ for $k=0,\ldots,m-1$. Recall that $S=P+\ii Q$ and that $u_i=(r_{ij} ~x_{ij})^T$. It follows from \eqref{m difference} that $$(\Delta P_{m-1,m-2} ~\Delta Q_{m-1,m-2})^T = -u_m\Delta \ell_{m,m-1}>0.$$ For any $k\in\{1,\ldots,m-1\}$, one has
	\begin{eqnarray*}
	\Delta S_{k-1,k-2} = \Delta S_{k,k-1} - z_{k,k-1} \Delta \ell_{k,k-1}
	= \Delta S_{k,k-1} - z_{k,k-1} \frac{|S_{k,k-1}'|^2-|S_{k,k-1}|^2}{v_k},
	\end{eqnarray*}
which is equivalent to
	\begin{equation*}\label{}
	\begin{pmatrix} \Delta P_{k-1,k-2} \\ \Delta Q_{k-1,k-2} \end{pmatrix}
	 = B_k
	 \begin{pmatrix} \Delta P_{k,k-1} \\ \Delta Q_{k,k-1} \end{pmatrix}
	\end{equation*}
where
	$$B_k = I-\frac{2}{v_k}\begin{pmatrix} r_{k,k-1} \\ x_{k,k-1} \end{pmatrix}
	\left(\frac{P_{k,k-1}+P_{k,k-1}'}{2} \quad \frac{Q_{k,k-1}+Q_{k,k-1}'}{2}\right).$$
Hence, one has
	$$(\Delta P_{k-1,k-2} ~\Delta Q_{k-1,k-2})^T = -B_kB_{k+1}\cdots B_{m-1} u_m\Delta\ell_{m,m-1}$$ 
for $k=1,\ldots,m$. To show that $\Delta S_{k,k-1}>0$ for $k=0,\ldots,m-1$, it suffices to show that $B_k\cdots B_{m-1}u_m>0$ for $k=1,\ldots,m$.

C1 implies that $\underline{A}_s\cdots \underline{A}_{t-1}u_t>0$ for $1\leq s\leq t\leq m$. One also has $B_k-\underline{A}_k = u_kb_k^T$ where
	\begin{equation*}
	b_k = \left(
	\frac{2\hat{P}_{k,k-1}^+(\overline{p})}{\underline{v}_k}    -    \frac{P_{k,k-1}+P_{k,k-1}'}{v_k} \qquad
	\frac{2\hat{Q}_{k,k-1}^+(\overline{q})}{\underline{v}_k}    -    \frac{Q_{k,k-1}+Q_{k,k-1}'}{v_k}
	\right)^T\geq0
	\end{equation*}
for $k=1,\ldots,m-1$. To show that $B_k\cdots B_{m-1}u_m>0$ for $k=1,\ldots,m$, we prove the following lemma.

	\begin{lemma}\label{lemma: dynamical system}
	Given $m\geq1$ and $d\geq1$. Let $\underline{A}_1,\ldots,\underline{A}_{m-1}, A_1,\ldots,A_{m-1}\in\mathbb{R}^{d\times d}$ 	and $u_1,\ldots,u_m\in\mathbb{R}^d$ satisfy
	\begin{itemize}
	\item $\underline{A}_s\cdots \underline{A}_{t-1}u_t>0$ when $1\leq s\leq t \leq m$;
	\item there exists $b_k\in\mathbb{R}^d$ that satisfies $b_k\geq0$ and $A_k-\underline{A}_k=u_kb_k^T$, for $k=1,\ldots,m-1$.
	\end{itemize}
	Then 
	\begin{equation}\label{dynamical}
	A_s\cdots A_{t-1}u_t>0
	\end{equation}
when $1\leq s\leq t \leq m$.
	\end{lemma}

\begin{proof}
We prove that \eqref{dynamical} holds when $1\leq t\leq s\leq m$ by mathematical induction on $t-s$.
\begin{itemize}
\item[i)] When $t-s=0$, one has
    $A_s\cdots A_{t-1}u_t = u_t = \underline{A}_s\cdots \underline{A}_{t-1}u_t > 0.$
\item[ii)] Assume that \eqref{dynamical} holds when $t-s=0,1,\ldots, K$ ($0\leq K\leq m-2$). When $t-s=K+1$, one has
    \begin{eqnarray*}
    A_s\cdots A_k\underline{A}_{k+1}\cdots \underline{A}_{t-1}u_t
    &=& A_s\cdots A_{k-1}\underline{A}_{k}\underline{A}_{k+1}\cdots \underline{A}_{t-1} u_t
    + A_s\cdots A_{k-1}(A_k-\underline{A}_{k})\underline{A}_{k+1}\cdots \underline{A}_{t-1} u_t\\
    &=& A_s\cdots A_{k-1}\underline{A}_{k}\cdots \underline{A}_{t-1} u_t
    + A_s\cdots A_{k-1}u_kb_k^T\underline{A}_{k+1}\cdots \underline{A}_{t-1} u_t\\
    &=& A_s\cdots A_{k-1}\underline{A}_{k}\cdots \underline{A}_{t-1} u_t 
    + \left(b_k^T\underline{A}_{k+1}\cdots \underline{A}_{t-1} u_t\right)A_s\cdots A_{k-1}u_k
    \end{eqnarray*}
for $k=s, \ldots, t-1$. Since $b_k\geq0$ and $\underline{A}_{k+1}\cdots \underline{A}_{t-1} u_t>0$, the term $b_k^T\underline{A}_{k+1}\cdots \underline{A}_{t-1} u_t\geq0$. According to induction hypothesis, $A_s\cdots A_{k-1} u_k>0$. Hence,
   \begin{eqnarray*}
    A_s\cdots A_k\underline{A}_{k+1}\cdots \underline{A}_{t-1}u_t
    \geq A_s\cdots A_{k-1}\underline{A}_{k}\cdots \underline{A}_{t-1} u_t
    \end{eqnarray*}
for $k=s, \ldots, t-1$. By substituting $k=t-1,\ldots,s$ in turn, one obtains
    \begin{eqnarray*}
    A_s\cdots A_{t-1}u_t
    \geq A_s\cdots A_{t-2}\underline{A}_{t-1}u_t
    \geq \cdots \geq \underline{A}_s \cdots \underline{A}_{t-1} u_t > 0,
    \end{eqnarray*}
i.e., \eqref{dynamical} holds when $t-s=K+1$.
\end{itemize}
According to (i) and (ii), \eqref{dynamical} holds when $t-s=0,\ldots,m-1$. This completes the proof of Lemma \ref{lemma: dynamical system}.
\end{proof}
Lemma \ref{lemma: dynamical system} implies that $B_s\cdots B_{t-1}u_t>0$ when $1\leq s\leq t\leq m$. In particular, $B_k\cdots B_{m-1}u_m>0$ for $k=1,\ldots, m$, and therefore $\Delta S_{k,k-1}>0$ for $k=0,\ldots,m-1$.

Next show that $\Delta S_{k,k-1}>0$ for $k=0,\ldots,m-1$ implies $v'\geq v$. Note that $\Delta S_{ij}=0$ when $(i,j)\notin\hP_{m-1}$ and $\Delta \ell_{ij}=0$ when $(i,j)\notin\hP_m$. It follows from \eqref{OPF v} that
	$$\Delta v_i-\Delta v_j = 2\re(\bar{z}_{ij} \Delta S_{ij}) - |z_{ij}|^2 \Delta \ell_{ij} = 0$$
when $(i,j)\notin\hP_m$. When $(i,j)\in\hP_m$, one has $(i,j)=(k,k-1)$ for some $k\in\{1,\ldots,m\}$, and therefore
	\begin{eqnarray*}
	\Delta v_i-\Delta v_j &=& 2\re(\bar{z}_{k,k-1} \Delta S_{k,k-1}) - |z_{k,k-1}|^2 \Delta \ell_{k,k-1} \\
	&\geq& \re(\bar{z}_{k,k-1} \Delta S_{k,k-1}) - |z_{k,k-1}|^2 \Delta \ell_{k,k-1} \\
	&=& \re(\bar{z}_{k,k-1} (\Delta S_{k,k-1}-z_{k,k-1}\Delta \ell_{k,k-1})) \\
	&=& \re(\bar{z}_{k,k-1} \Delta S_{k-1,k-2}) > 0.
	\end{eqnarray*}
Hence, $\Delta v_i\geq\Delta v_j$ whenever $(i,j)\in\hE$. Add the inequalities over path $\hP_i$ to obtain $\Delta v_i\geq\Delta v_0=0$ for $i\in\hN^+$, i.e., $v'\geq v$. This completes the proof of Claim \ref{claim: C1 tree}.

\section{Proof of Proposition \ref{prop: smaller injections}}\label{app: smaller injections}
Let $\underline{A}$ and $\underline{A}'$ denote the matrices with respect to $(\overline{p},\overline{q})$ and $(\overline{p}',\overline{q}')$ respectively, i.e., denote
	$$\underline{A}_i'=I - \frac{2}{\underline{v}_i} u_i \left(\hat{P}^+_{ij}(\overline{p}') ~\hat{Q}_{ij}^+(\overline{q}')\right)
	\text{ and }
	\underline{A}_i = I - \frac{2}{\underline{v}_i} u_i \left(\hat{P}^+_{ij}(\overline{p}) ~\hat{Q}_{ij}^+(\overline{q})\right)$$
for $(i,j)\in\hE$. When $(\overline{p},\overline{q})\leq(\overline{p}',\overline{q}')$,  one has $\underline{A}_{l_k}-\underline{A}_{l_k}'=u_{l_k}b_{l_k}^T$ where
	$$b_{l_k}=\frac{2}{\underline{v}_{l_k}}\begin{pmatrix}
	\hat{P}_{l_kl_{k-1}}^+(\overline{p}')     -     \hat{P}_{l_kl_{k-1}}^+(\overline{p}) \\
	\hat{Q}_{l_kl_{k-1}}^+(\overline{q}')     -     \hat{Q}_{l_kl_{k-1}}^+(\overline{q})
	\end{pmatrix}\geq0$$
for any $l\in\hL$ and any $k\in\{1\ldots, n_l\}$.

If $\underline{A}_{l_s}'\cdots \underline{A}_{l_{t-1}}'u_{l_t}>0$ for any $l\in\hL$ and any $s,t$ such that $1\leq s\leq t\leq n_l$, then it follows from Lemma \ref{lemma: dynamical system} that $\underline{A}_{l_s}\cdots \underline{A}_{l_{t-1}}u_{l_t}>0$ for any $l\in\hL$ any $s,t$ such that $1\leq s\leq t\leq n_l$. This completes the proof of Proposition \ref{prop: smaller injections}.

\section{Proof of Theorem \ref{thm: unique}}\label{app: thm unique}
Assume that $f_i$ is convex for $i\in \hN$, that $\mathcal{S}_i$ is convex for $i\in\hN^+$, that SOCP-m is exact, and that SOCP-m has at least one solution. Let $\tilde{w}=(\tilde{s},\tilde{S},\tilde{v},\tilde{\ell},\tilde{s}_0)$ and $\hat{w}=(\hat{s},\hat{S},\hat{v},\hat{\ell},\hat{s}_0)$ denote two arbitrary SOCP-m solutions. It suffices to show that $\tilde{w}=\hat{w}$.

Since SOCP-m is exact, $\tilde{v}_i\tilde{\ell}_{ij}=|\tilde{S}_{ij}|^2$ and $\hat{v}_i\hat{\ell}_{ij}=|\hat{S}_{ij}|^2$ for $(i,j)\in\hE$. Define $w\eqdef(\tilde{w}+\hat{w})/2$. Since SOCP-m is convex, $w$ also solves SOCP-m. Hence, $v_i\ell_{ij}=|S_{ij}|^2$ for $(i,j)\in\hE$. Substitute $v_i=(\tilde{v}_i+\hat{v}_i)/2$, $\ell_{ij} = (\tilde{\ell}_{ij} + \hat{\ell}_{ij})/2$, and $S_{ij}=(\tilde{S}_{ij}+\hat{S}_{ij})/2$ to obtain
	\begin{equation*}
	\hat{S}_{ij}\tilde{S}_{ij}^H + \tilde{S}_{ij}\hat{S}_{ij}^H = \hat{v}_i\tilde{\ell}_{ij} + \tilde{v}_i\hat{\ell}_{ij}
	\end{equation*}
for $(i,j)\in\hE$. The right hand side
	\begin{equation*}
	\hat{v}_i\tilde{\ell}_{ij} + \tilde{v}_i\hat{\ell}_{ij}
	=\hat{v}_i\frac{|\tilde{S}_{ij}|^2}{\tilde{v}_i} + \tilde{v}_i\frac{|\hat{S}_{ij}|^2}{\hat{v}_i}
	\geq 2|\tilde{S}_{ij}||\hat{S}_{ij}|,
	\end{equation*}
and the equality is attained if and only if $|\tilde{S}_{ij}|/\tilde{v}_i=|\hat{S}_{ij}|/\hat{v}_i$. The left hand side
	\begin{equation*}
	\hat{S}_{ij}\tilde{S}_{ij}^H + \tilde{S}_{ij}\hat{S}_{ij}^H \leq 2|\tilde{S}_{ij}||\hat{S}_{ij}|,
	\end{equation*}
and the equality is attained if and only if $\angle \hat{S}_{ij}=\angle\tilde{S}_{ij}$. Hence, $\tilde{S}_{ij}/\tilde{v}_i=\hat{S}_{ij}/\hat{v}_i$ for $(i,j)\in\hE$.

Introduce $\hat{v}_0:=\tilde{v}_0:=v_0$ and define $\eta_i:=\hat{v}_i/\tilde{v}_i$ for $i\in\hN$, then $\eta_0=1$ and $\hat{S}_{ij}=\eta_i\tilde{S}_{ij}$ for $(i,j)\in\hE$. Hence,
	$$\hat{\ell}_{ij} = \frac{|\hat{S}_{ij}|^2}{\hat{v}_i} = \frac{|\eta_i\tilde{S}_{ij}|^2}{\eta_i\tilde{v}_i} = \eta_i\frac{|\tilde{S}_{ij}|^2}{\tilde{v}_i} = \eta_i\tilde{\ell}_{ij}$$
and therefore
	$$\eta_j = \frac{\hat{v}_j}{\tilde{v}_j} = \frac{\hat{v}_i-2\re(z_{ij}^H\hat{S}_{ij})+|z_{ij}|^2\hat{\ell}_{ij}}
	{\tilde{v}_i-2\re(z_{ij}^H\tilde{S}_{ij})+|z_{ij}|^2\tilde{\ell}_{ij}}=\eta_i$$
for $(i,j)\in\hE$. Since the network $(\hN,\hE)$ is connected, $\eta_i=\eta_0=1$ for $i\in\hN$. This implies $\hat{w}=\tilde{w}$ and completes the proof of Theorem \ref{thm: unique}.

\section{Proof of Theorem \ref{lemma: weaker}}\label{app: weaker}
Theorem \ref{lemma: weaker} follows from Claims \ref{claim: no reverse}--\ref{claim: pre C1}.

	\begin{claim}\label{claim: no reverse}
	Assume that there exists $\overline{p}_i$ and $\overline{q}_i$ such that $\mathcal{S}_i\subseteq\{s\in\mathbb{C} ~|~ \re(s)\leq\overline{p}_i,~\im(s)\leq\overline{q}_i\}$ for $i\in\hN^+$. If $\hat{P}_{ij}(\overline{p})\leq0$, $\hat{Q}_{ij}(\overline{q})\leq0$ for any $(i,j)\in\hE$ such that $i\notin \hL$, then C1 holds.
	\end{claim}
	
	\begin{proof}
	If $\hat{P}_{ij}(\overline{p})\leq0$, $\hat{Q}_{ij}(\overline{q})\leq0$ for any $(i,j)\in\hE$ such that $i\notin \hL$, then $\underline{A}_{l_k}=I$ for any $l\in\hL$ and any $k\in\{1\ldots, n_l-1\}$. It follows that $\underline{A}_{l_s}\cdots\underline{A}_{l_{t-1}}u_{l_t}=u_{l_t}>0$ for any $l\in\hL$ and any $s,t$ such that $1\leq s\leq t\leq n_l$, i.e., C1 holds.
	\end{proof}

	\begin{claim}\label{claim: uniform}
	Assume that there exists $\overline{p}_i$ and $\overline{q}_i$ such that $\mathcal{S}_i\subseteq\{s\in\mathbb{C} ~|~ \re(s)\leq\overline{p}_i,~\im(s)\leq\overline{q}_i\}$ for $i\in\hN^+$. If $r_{ij}/x_{ij}=r_{jk}/x_{jk}$ for any $(i,j),(j,k)\in\hE$, and $\underline{v}_i-2r_{ij}\hat{P}_{ij}^+(\overline{p})-2x_{ij}\hat{Q}_{ij}^+(\overline{q})>0$ for any $(i,j)\in\hE$ such that $i\notin \hL$, then C1 holds.
	\end{claim}
	
	\begin{proof}
	Assume that $r_{ij}/x_{ij}=r_{jk}/x_{jk}$ for any $(i,j),(j,k)\in\hE$, and that $\underline{v}_i-2r_{ij}\hat{P}_{ij}^+(\overline{p})-2x_{ij}\hat{Q}_{ij}^+(\overline{q})>0$ for any $(i,j)\in\hE$ such that $i\notin \hL$. Fix an arbitrary $l\in\hL$, and assume $l_k=k$ for $k=0,\ldots,n_l$ without loss of generality. Fix an arbitrary $t\in\{1,\ldots,n_l\}$, and define $(\alpha_s ~ \beta_s)^T:=\underline{A}_s\cdots \underline{A}_{t-1}u_t$ for $s=1,\ldots,t$. Then it suffices to prove that $\alpha_s>0$ and $\beta_s>0$ for $s=1,\ldots, t$. In particular, we prove
	\begin{equation}\label{H1}
	\alpha_s>0, ~ \beta_s>0, ~ \alpha_s/\beta_s=r_{10}/x_{10}
	\end{equation}
inductively for $s=t,t-1,\ldots,1$. Define $\eta:=r_{10}/x_{10}$ and note that $r_{ij}/x_{ij}=\eta$ for $(i,j)\in \hE$.
	\begin{itemize}
	\item[i)] When $s=t$, one has $\alpha_s=r_{t,t-1}$, $\beta_s=x_{t,t-1}$, and $\alpha_s/\beta_s=\eta$. Therefore \eqref{H1} holds.
	\item[ii)] Assume that \eqref{H1} holds for $s=k$ ($2\leq k\leq t$), then $(\alpha_k ~\beta_k)^T=cu_{k-1}$ for some $c\in\mathbb{R}$. It follows that
	\begin{eqnarray*}
	\begin{pmatrix}	\alpha_{k-1}\\ \beta_{k-1}	\end{pmatrix}
	&=& \left[I-   \frac{2}{\underline{v}_{k-1}}  u_{k-1}\begin{pmatrix}
	\hat{P}_{k-1,k-2}^+(\overline{p}) & \hat{Q}_{k-1,k-2}^+(\overline{q})
	\end{pmatrix}  \right]
	\begin{pmatrix} \alpha_k \\ \beta_k \end{pmatrix}\\
	&=& \left(1- \frac{2}{\underline{v}_{k-1}}
	\begin{pmatrix}
	\hat{P}_{k-1,k-2}^+(\overline{p}) & \hat{Q}_{k-1,k-2}^+(\overline{q})
	\end{pmatrix} u_{k-1}\right) \begin{pmatrix} \alpha_k \\ \beta_k \end{pmatrix} \\
	&=& \frac{1}{\underline{v}_{k-1}} \left(\underline{v}_{k-1}-
	2 r_{k-1,k-2}\hat{P}_{k-1,k-2}^+(\overline{p}) - 2x_{k-1,k-2}\hat{Q}_{k-1,k-2}^+(\overline{q})       \right)
	\begin{pmatrix} \alpha_k \\ \beta_k \end{pmatrix} > 0
	\end{eqnarray*}
and $\alpha_{k-1}/\beta_{k-1}=\alpha_k/\beta_k=\eta$. Hence, \eqref{H1} holds for $s=k-1$.
	\end{itemize}
According to (i) and (ii), \eqref{H1} holds for $s=t,t-1\ldots,1$. This completes the proof of Claim \ref{claim: uniform}.
	\end{proof}
	
	\begin{claim}\label{claim: P}
	Assume that there exists $\overline{p}_i$ and $\overline{q}_i$ such that $\mathcal{S}_i\subseteq\{s\in\mathbb{C} ~|~ \re(s)\leq\overline{p}_i,~\im(s)\leq\overline{q}_i\}$ for $i\in\hN^+$. If $r_{ij}/x_{ij}\geq r_{jk}/x_{jk}$ for any $(i,j),(j,k)\in\hE$, and $\hat{P}_{ij}(\overline{p})\leq0$, $\underline{v}_i-2x_{ij}\hat{Q}_{ij}^+(\overline{q})>0$ for any $(i,j)\in\hE$ such that $i\notin \hL$, then C1 holds.
	\end{claim}
	
	\begin{proof}
	Assume that $r_{ij}/x_{ij}\geq r_{jk}/x_{jk}$ for any $(i,j),(j,k)\in\hE$, and that $\hat{P}_{ij}(\overline{p})\leq0$, $\underline{v}_i-2x_{ij}\hat{Q}_{ij}^+(\overline{q})>0$ for any $(i,j)\in\hE$ such that $i\notin \hL$. Fix an arbitrary $l\in\hL$, and assume $l_k=k$ for $k=0,\ldots,n_l$ without loss of generality. Fix an arbitrary $t\in\{1,\ldots,n_l\}$, and define $(\alpha_s ~ \beta_s)^T:=\underline{A}_s\cdots \underline{A}_{t-1}u_t$ for $s=1,\ldots,t$. Then it suffices to prove that $\alpha_s>0$ and $\beta_s>0$ for $s=1,\ldots, t$. In particular, we prove
	\begin{equation}\label{H2}
	\alpha_s>0, ~ \beta_s>0, ~ \alpha_s/\beta_s\geq r_{t,t-1}/x_{t,t-1}
	\end{equation}
inductively for $s=t,t-1,\ldots,1$.
	\begin{itemize}
	\item[i)] When $s=t$, one has $\alpha_s=r_{t,t-1}$, $\beta_s=x_{t,t-1}$, and $\alpha_s/\beta_s=r_{t,t-1}/x_{t,t-1}$. Therefore \eqref{H2} holds.
	\item[ii)] Assume that \eqref{H2} holds for $s=k$ ($2\leq k\leq t$). Noting that $\hat{P}_{k-1,k-2}^+(\overline{p})=0$, one has
	\begin{eqnarray*}
	\begin{pmatrix}	\alpha_{k-1}\\ \beta_{k-1}	\end{pmatrix}
	&=& \left[I-   \frac{2}{\underline{v}_{k-1}}  u_{k-1}\begin{pmatrix}
	\hat{P}_{k-1,k-2}^+(\overline{p}) & \hat{Q}_{k-1,k-2}^+(\overline{q})
	\end{pmatrix}  \right]
	\begin{pmatrix} \alpha_k \\ \beta_k \end{pmatrix}\\
	&=& \begin{pmatrix} \alpha_k \\ \beta_k \end{pmatrix} - \frac{2}{\underline{v}_{k-1}}  u_{k-1}
	\hat{Q}_{k-1,k-2}^+(\overline{q})    \beta_k.
	\end{eqnarray*}
Hence, $\beta_{k-1} = \frac{1}{\underline{v}_{k-1}}\left(\underline{v}_{k-1}-2x_{k-1,k-2}\hat{Q}_{k-1,k-2}^+(\overline{q})\right) \beta_k>0$. Then,
	\begin{eqnarray*}
	\alpha_{k-1} &=& \alpha_k - \frac{2r_{k-1,k-2}\hat{Q}_{k-1,k-2}^+(\overline{q})}{\underline{v}_{k-1}}\beta_k\\
	&\geq& \left(\frac{r_{t,t-1}}{x_{t,t-1}}- \frac{2r_{k-1,k-2}\hat{Q}_{k-1,k-2}^+(\overline{q})}{\underline{v}_{k-1}}\right)\beta_k \\
    &\geq& \frac{r_{t,t-1}}{x_{t,t-1}} \left(1- \frac{2x_{k-1,k-2}\hat{Q}_{k-1,k-2}^+(\overline{q})}{\underline{v}_{k-1}}\right)\beta_k
	~=~ \frac{r_{t,t-1}}{x_{t,t-1}}\beta_{k-1} > 0.
	\end{eqnarray*}
The second inequality is due to $r_{k-1,k-2}/x_{k-1,k-2}\leq r_{t,t-1}/x_{t,t-1}$. Hence, \eqref{H2} holds for $s=k-1$.
	\end{itemize}
	According to (i) and (ii), \eqref{H2} holds for $s=t,t-1,\ldots,1$. This completes the proof of Claim \ref{claim: P}.
	\end{proof}	

	\begin{claim}\label{claim: Q}
	Assume that there exists $\overline{p}_i$ and $\overline{q}_i$ such that $\mathcal{S}_i\subseteq\{s\in\mathbb{C} ~|~ \re(s)\leq\overline{p}_i,~\im(s)\leq\overline{q}_i\}$ for $i\in\hN^+$. If $r_{ij}/x_{ij}\leq r_{jk}/x_{jk}$ for any $(i,j),(j,k)\in\hE$, and $\hat{Q}_{ij}(\overline{q})\leq0$, $\underline{v}_i-2r_{ij}\hat{P}_{ij}^+(\overline{p})>0$ for any $(i,j)\in\hE$ such that $i\notin \hL$, then C1 holds.
	\end{claim}
	
	\begin{proof}
	The proof of Claim \ref{claim: Q} is similar to that of Claim \ref{claim: P}, and omitted for brevity.
	\end{proof}
	
	\begin{claim}\label{claim: pre C1}
	Assume that there exists $\overline{p}_i$ and $\overline{q}_i$ such that $\mathcal{S}_i\subseteq\{s\in\mathbb{C} ~|~ \re(s)\leq\overline{p}_i,~\im(s)\leq\overline{q}_i\}$ for $i\in\hN^+$. If
	\begin{equation}\label{pre C1}
	\begin{pmatrix}
	\displaystyle \prod_{(k,l)\in \hP_j}\left( 1-\frac{2r_{kl}\hat{P}_{kl}^+(\overline{p})}{\underline{v}_k} \right) &
	\displaystyle -\sum_{(k,l)\in \hP_j}\frac{2r_{kl}\hat{Q}_{kl}^+(\overline{q})}{\underline{v}_k} \\
	\displaystyle -\sum_{(k,l)\in \hP_j}\frac{2x_{kl}\hat{P}_{kl}^+(\overline{p})}{\underline{v}_k} &
	\displaystyle \prod_{(k,l)\in \hP_j}\left( 1-\frac{2x_{kl}\hat{Q}_{kl}^+(\overline{q})}{\underline{v}_k} \right)
	\end{pmatrix}
	\begin{pmatrix}
	r_{ij} \\ x_{ij}
	\end{pmatrix} >0
	\end{equation}
for $(i,j)\in\hE$, then C1 holds.
	\end{claim}
The following lemma is used in the proof of Claim \ref{claim: pre C1}.

\begin{lemma}\label{lemma: matrix multiplication}
	Given $i\geq1$; $c$, $d$, $e$, $f\in\mathbb{R}^i$ such that $0<c\leq1$, $d\geq0$, $e\geq0$, and $0<f\leq1$ componentwise; and $u\in\mathbb{R}^2$ that satisfies $u>0$. If
	\begin{equation}\label{lemma 7 condition}
	\begin{pmatrix}
	\displaystyle \prod_{j=1}^i c_j & \displaystyle -\sum_{j=1}^i d_j\\
	\displaystyle -\sum_{j=1}^i e_j & \displaystyle \prod_{j=1}^i f_j
	\end{pmatrix}
	u>0,
	\end{equation}
then
    \begin{equation}\label{lemma 7 result}
    \begin{pmatrix}
	c_j & -d_j\\
	-e_j & f_j
	\end{pmatrix}\cdots
    \begin{pmatrix}
	c_i & -d_i\\
	-e_i & f_i
	\end{pmatrix} u > 0
    \end{equation}
for $j=1,\ldots,i$.
	\end{lemma}
	
	\begin{proof}
Lemma \ref{lemma: matrix multiplication} can be proved by mathematical induction on $i$.
	\begin{itemize}
	\item[i)] When $i=1$, Lemma \ref{lemma: matrix multiplication} is trivial.
	\item[ii)] Assume that Lemma \ref{lemma: matrix multiplication} holds for $i=K$ ($K\geq1$). When $i=K+1$, if
	\begin{eqnarray*}
	\begin{pmatrix}
	\displaystyle \prod_{j=1}^{i}c_j & \displaystyle -\sum_{j=1}^{i}d_j\\
	\displaystyle -\sum_{j=1}^{i}e_j & \displaystyle \prod_{j=1}^{i}f_j
	\end{pmatrix}
	u > 0,
	\end{eqnarray*}
one can prove that \eqref{lemma 7 result} holds for $j=1,\ldots,K+1$ as follows.

First prove that \eqref{lemma 7 result} holds for $j=2,\ldots,K+1$. The idea is to construct some $c',d',e',f'\in\mathbb{R}^K$ and apply the induction hypothesis. The construction is
    \begin{eqnarray*}
    c' = (c_2,~c_3,~\ldots,~c_{K+1}),&& d' = (d_2,~d_3,~\ldots,~d_{K+1}),\\
    e' = (e_2,~e_3,~\ldots,~e_{K+1}),&&f' = (f_2,~f_3,~\ldots,~f_{K+1}).
    \end{eqnarray*}
Clearly, $c',d',e',f'$ satisfies $0<c'\leq1$, $d'\geq0$, $e'\geq0$, $0<f'\leq1$ componentwise and
	\begin{eqnarray*}
	\begin{pmatrix}
	\displaystyle \prod_{j=1}^Kc_j' & \displaystyle -\sum_{j=1}^Kd_j'\\
	\displaystyle -\sum_{j=1}^Ke_j' & \displaystyle \prod_{j=1}^Kf_j'
	\end{pmatrix}
	u
    &\!\!\!=& \!\!\!\begin{pmatrix}
	\displaystyle \prod_{j=2}^{K+1}c_j & \displaystyle -\sum_{j=2}^{K+1}d_j\\
	\displaystyle -\sum_{j=2}^{K+1}e_j & \displaystyle \prod_{j=2}^{K+1}f_j
	\end{pmatrix}
	u \geq
	\begin{pmatrix}
	\displaystyle \prod_{j=1}^{K+1}c_j & \displaystyle -\sum_{j=1}^{K+1}d_j\\
	\displaystyle -\sum_{j=1}^{K+1}e_j & \displaystyle \prod_{j=1}^{K+1}f_j
	\end{pmatrix}
	u > 0.
	\end{eqnarray*}
Apply the induction hypothesis to obtain that
    $$\begin{pmatrix}
	c_j' & -d_j'\\
	-e_j' & f_j'
	\end{pmatrix}\cdots
    \begin{pmatrix}
	c_K' & -d_K'\\
	-e_K' & f_K'
	\end{pmatrix} u > 0$$
for $j=1,\ldots,K$, i.e., \eqref{lemma 7 result} holds for $j=2,\ldots,K+1$.

Next prove that \eqref{lemma 7 result} holds for $j=1$. The idea is still to construct some $c',d',e',f'\in\mathbb{R}^K$ and apply the induction hypothesis. The construction is
    \begin{eqnarray*}
    c' = (c_1c_2,~c_3,~\ldots,~c_{K+1}), && d' = (d_1+d_2,~d_3,~\ldots,~d_{K+1}),\\
    e' = (e_1+e_2,~e_3,~\ldots,~e_{K+1}), && f' = (f_1f_2,~f_3,~\ldots,~f_{K+1}).
    \end{eqnarray*}
Clearly, $c',d',e',f'$ satisfies $0<c'\leq1$, $d'\geq0$, $e'\geq0$, $0<f'\leq1$ componentwise and
    \begin{eqnarray*}
	\begin{pmatrix}
	\displaystyle \prod_{j=1}^Kc_j' & \displaystyle -\sum_{j=1}^Kd_j'\\
	\displaystyle -\sum_{j=1}^Ke_j' & \displaystyle \prod_{j=1}^Kf_j'
	\end{pmatrix}
	u
    = \begin{pmatrix}
	\displaystyle \prod_{j=1}^{K+1}c_j & \displaystyle -\sum_{j=1}^{K+1}d_j\\
	\displaystyle -\sum_{j=1}^{K+1}e_j & \displaystyle \prod_{j=1}^{K+1}f_j
	\end{pmatrix}
	u
	>0.
	\end{eqnarray*}
Apply the induction hypothesis to obtain
    \begin{eqnarray*}
    v_2':=\begin{pmatrix}
	c_2' & -d_2'\\
	-e_2' & f_2'
	\end{pmatrix}\cdots
    \begin{pmatrix}
	c_K' & -d_K'\\
	-e_K' & f_K'
	\end{pmatrix} u > 0, \quad
    v_1':=\begin{pmatrix}
	c_1' & -d_1'\\
	-e_1' & f_1'
	\end{pmatrix}\cdots
    \begin{pmatrix}
	c_K' & -d_K'\\
	-e_K' & f_K'
	\end{pmatrix} u > 0.
    \end{eqnarray*}
It follows that
    \begin{eqnarray*}
    \begin{pmatrix}
	c_1 & -d_1\\
	-e_1 & f_1
	\end{pmatrix}\cdots
    \begin{pmatrix}
	c_{K+1} & -d_{K+1}\\
	-e_{K+1} & f_{K+1}
	\end{pmatrix} u
    &=& \begin{pmatrix}
	c_1 & -d_1\\
	-e_1 & f_1
	\end{pmatrix}
    \begin{pmatrix}
	c_2 & -d_2\\
	-e_2 & f_2
	\end{pmatrix}
    \begin{pmatrix}
	c_3 & -d_3\\
	-e_3 & f_3
	\end{pmatrix}
\cdots
    \begin{pmatrix}
	c_{K+1} & -d_{K+1}\\
	-e_{K+1} & f_{K+1}
	\end{pmatrix} u \\
    &=& \begin{pmatrix}
	c_1 & -d_1\\
	-e_1 & f_1
	\end{pmatrix}
    \begin{pmatrix}
	c_2 & -d_2\\
	-e_2 & f_2
	\end{pmatrix} v_2' \\
    &=& \begin{pmatrix}
	c_1c_2+d_1e_2 & -c_1d_2-d_1f_2\\
	-e_1c_2-f_1e_2 & f_1f_2+e_1d_2
	\end{pmatrix} v_2' \\
    &\geq& \begin{pmatrix}
	c_1c_2 & -d_2-d_1\\
	-e_1-e_2 & f_1f_2
	\end{pmatrix} v_2' \\
    &=& \begin{pmatrix}
	c_1' & -d_1'\\
	-e_1' & f_1'
	\end{pmatrix} v_2' ~=~ v_1'>0,
    \end{eqnarray*}
i.e., \eqref{lemma 7 result} holds for $j=1$.

To this end, we have proved that \eqref{lemma 7 result} holds for $j=1,\ldots,K+1$, i.e., Lemma \ref{lemma: matrix multiplication} also holds for $i=K+1$.
	\end{itemize}
According to (i) and (ii), Lemma \ref{lemma: matrix multiplication} holds for $i\geq1$.
	\end{proof}

\noindent{\it Proof of Claim \ref{claim: pre C1}.}
Fix an arbitrary $l\in\hL$, and assume $l_k=k$ for $k=0,\ldots,n_l$ without loss of generality. Fix an arbitrary $t\in\{1,\ldots,n_l\}$, then it suffices to prove that $\underline{A}_s\cdots\underline{A}_{t-1}u_t>0$ for $s=1,\ldots,t$. Denote $r_k:=r_{k,k-1}$ and $S_k:=S_{k,k-1}$ for $k=1,\ldots,t$ for brevity.

Substitute $(i,j)=(k,k-1)$ in \eqref{pre C1} to obtain
	\begin{equation}\label{matrix linear}
	\begin{pmatrix}
	\displaystyle \prod_{s=1}^{k-1}\left( 1-\frac{2r_s\hat{P}_s^+}{\underline{v}_s} \right) &
	\displaystyle -\sum_{s=1}^{k-1}\frac{2r_s\hat{Q}_s^+}{\underline{v}_s} \\
	\displaystyle -\sum_{s=1}^{k-1}\frac{2x_s\hat{P}_s^+}{\underline{v}_s} &
	\displaystyle \prod_{s=1}^{k-1}\left( 1-\frac{2x_s\hat{Q}_s^+}{\underline{v}_s} \right)
	\end{pmatrix}
	\begin{pmatrix}
	r_k \\ x_k
	\end{pmatrix} >0
	\end{equation}
for $k=1,\ldots,t$. Hence,	
	$$\prod_{s=1}^{k-1}\left( 1-\frac{2r_s\hat{P}_s^+}{\underline{v}_s} \right) r_k > \sum_{s=1}^{k-1} \frac{2r_s\hat{Q}_s^+(\overline{q})}{\underline{v}_s} x_k \geq0 $$
for $k=1,\ldots,t$. It follows that $1-2r_k\hat{P}_k^+/\underline{v}_k>0$ for $k=1,\ldots,t-1$. Similarly, $1-2x_k\hat{Q}_k^+/\underline{v}_k>0$ for $k=1,\ldots,t-1$. Then, substitute $k=t$ in \eqref{matrix linear} and apply Lemma \ref{lemma: matrix multiplication} to obtain
	$$\begin{pmatrix}
	\displaystyle 1-\frac{2r_s\hat{P}_s^+}{\underline{v}_s} &
	\displaystyle -\frac{2r_s\hat{Q}_s^+}{\underline{v}_s} \\
	\displaystyle -\frac{2x_s\hat{P}_s^+}{\underline{v}_s} &
	\displaystyle 1-\frac{2x_s\hat{Q}_s^+}{\underline{v}_s}
	\end{pmatrix}\cdots
	\begin{pmatrix}
	\displaystyle 1-\frac{2r_{t-1}\hat{P}_{t-1}^+(\overline{p})}{\underline{v}_{t-1}} &
	\displaystyle -\frac{2r_{t-1}\hat{Q}_{t-1}^+(\overline{q})}{\underline{v}_{t-1}} \\
	\displaystyle -\frac{2x_{t-1}\hat{P}_{t-1}^+(\overline{p})}{\underline{v}_{t-1}} &
	\displaystyle 1-\frac{2x_{t-1}\hat{Q}_{t-1}^+(\overline{p})}{\underline{v}_{t-1}}
	\end{pmatrix}
	\begin{pmatrix}
	r_t \\ x_t
	\end{pmatrix} >0$$
for $s=1,\ldots,t$, i.e., $\underline{A}_s\cdots \underline{A}_{t-1}u_t>0$ for $s=1,\ldots,t$. This completes the proof of Claim \ref{claim: pre C1}. $\hfill\Box$

\end{document}